\documentclass{elsarticle}
\usepackage{amsmath,amssymb,amsfonts,amscd}
\journal{}%%Indagationes Mathematicae}

\newtheorem{theorem}{Theorem}[section]
\newtheorem{definition}[theorem]{Definition}
\newtheorem{lemma}[theorem]{Lemma}
\newtheorem{remark}[theorem]{Remark}
\newproof{proof}{\it Proof}

\begin{document}

\begin{frontmatter}



\title{Analytic operator-valued generalized Feynman  integral on function space}

\author{Jae Gil Choi}
\ead{jgchoi@dankook.ac.kr}

%\address[label3]
\address{School of General Education,  
                 Dankook University,
                 Cheonan 31116, 
                 Republic of Korea}

%%\affiliation{organization={},%Department and Organization

\begin{abstract}
In this paper  an  analytic operator-valued 
generalized Feynman integral  was studied on a very general Wiener  space $C_{a,b}[0,T]$.
The general Wiener  space $C_{a,b}[0,T]$ is a function space which is induced by the 
generalized Brownian motion process
associated with continuous functions $a$ and $b$.
The structure of the analytic operator-valued generalized Feynman integral is suggested and 
the existence of the analytic operator-valued generalized Feynman integral is
investigated  as an operator from $L^1(\mathbb R, \nu_{\delta,a})$ to $L^{\infty}(\mathbb R)$
where $\nu_{\delta,a}$ is a $\sigma$-finite measure on $\mathbb R$
given by  
\[
d\nu_{\delta,a}=\exp\{\delta \mathrm{Var}(a)u^2\} du, 
\]
where  $\delta>0$ and $\mathrm{Var}(a)$ denotes the total variation of the
mean function $a$ of the  generalized Brownian motion process.
It turns out in this paper that
the analytic operator-valued generalized Feynman  integrals
of functionals defined by the stochastic Fourier--Stieltjes transform 
of complex measures on the infinite dimensional Hilbert space 
$C_{a,b}'[0,T]$ are elements of the linear space 
\[
\bigcap_{\delta>0} \mathcal L( L^1(\mathbb R,\nu_{\delta,a}),L^{\infty}(\mathbb R)).
\]
\end{abstract}


%%\item Research highlight 1
%%\item Research highlight 2

\begin{keyword}
generalized Brownian motion process \sep 
analytic operator-valued generalized function space integral \sep 
analytic operator-valued generalized Feynman integral 

\vspace{.3cm}

\MSC[2010]
28C20 \sep %%Set functions and measures and integrals in infinite-dimensional spaces 
46G12 \sep %%Measures and integration on abstract linear spaces 
46B09      %%Probabilistic methods in Banach space theory  
\end{keyword}

\end{frontmatter}



%% \tableofcontents
%%%%%%%%%%%%%%%%%%%%%%%%%%%%%%%%%%%%%%%%%%%%%%%%%%%%%%%%%%%%%%%%%%
%%%%%%%%%%%%%%%%%%%%%%%%%%%%%%%%%%%%%%%%%%%%%%%%%%%%%%%%%%%%%%%%%%
%%%-----------------------[section]----------------------------%%%
%%%%%%%%%%%%%%%%%%%%%%%%%%%%%%%%%%%%%%%%%%%%%%%%%%%%%%%%%%%%%%%%%%
%%%%%%%%%%%%%%%%%%%%%%%%%%%%%%%%%%%%%%%%%%%%%%%%%%%%%%%%%%%%%%%%%%
%%  
\setcounter{equation}{0}
\section{Introduction}\label{sec:1}

\par
 Before giving a basic survey and a motivation for our topic we fix some notation.
Let $\mathbb C$, $\mathbb C_+$ and  $\widetilde{\mathbb{C}}_+$  denote 
the  set of  complex numbers,   complex numbers with positive real part
and nonzero complex numbers with nonnegative real part, respectively.
For all $\lambda \in \widetilde{\mathbb{C}}_+$, $\lambda^{1/2}
\equiv\sqrt{\lambda}$ (or $\lambda^{-1/2}$)
is always chosen to have positive real part.
Furthermore, let $C[0,T]$ denote  the space of  
real-valued continuous  functions  $x$ on $[0,T]$ and
let $C_0[0,T]$ denote those $x$ in $C[0,T]$ such that $x(0)=0$. 
The function space $C_0[0,T]$  is referred to as one-parameter Wiener space, 
and we let $m_w$ denote Wiener measure.
Given   two Banach spaces $X$ and $Y$,
let  $\mathcal L(X,Y)$ denote the space of continuous linear operators
from $X$ to $Y$.

Let $F$ be a $\mathbb C$-valued measurable functional on $C[0,T]$.
For $\lambda>0$, $\psi\in L^2(\mathbb R)$, and $\xi\in \mathbb R$, consider
the Wiener integral
\begin{equation}\label{eq:intro}
(I_{\lambda}(F)\psi)(\xi)\equiv \int_{C_0[0,T]}
F\big(\lambda^{-1/2}x+\xi\big)\psi\big(\lambda^{-1/2}x(T)+\xi\big)dm_w(x).
\end{equation}
In the application of the Feynman integral to quantum theory,
the function $\psi$ in \eqref{eq:intro} corresponds to the initial condition
associated with Schr\"odinger equation.

In \cite{CS68}, Cameron and Storvick considered the following natural  and  interesting questions.
Under what conditions on $F$
will the linear operator $I_{\lambda}(F)$ given by \eqref{eq:intro}
be an element of  $\mathcal L(L^2(\mathbb R),L^{2}(\mathbb R))$? If so,
does the operator valued function $\lambda \to I_{\lambda}(F)$ have an analytic extension,
write $I_{\lambda}^{\mathrm{an} }(F)$ (it is called the analytic operator-valued Wiener integral of $F$ with
parameter $\lambda$), to $\mathbb C_+$?
If so, for each nonzero real number $q$, does the limit
\[%%begin{equation}\label{eq:intro2}
J_{q}^{\mathrm{an} }(F) 
\equiv \lim_{\substack{\lambda\to-iq\\\lambda\in \mathbb C_+}}
I_{\lambda}^{\mathrm{an}}(F)
\]%%end{equation}
exist in some topological (normed) structure? The functional 
$J_{q}^{\mathrm{an}}(F)$ (if it exists) is called the  analytic operator-valued 
Feynman  integral of $F$ with parameter $q$.

Cameron and Storvick in \cite{CS68} introduced 
the analytic operator-valued function space 
``Feynman integral'' $J_{q}^{\mathrm{an}}(F)$, 
which mapped an $L^2(\mathbb R)$ function 
$\psi$ into an $L^2(\mathbb R)$ function $J_{q}^{\mathrm{an}}(F)\psi$. 
In \cite{CS68} and several subsequent papers 
\cite{CS70,CS73-1,JS70-1,JS70-2,JS71,JS73-1,JS73-2,JS74-1,JS74-2,JS75-1}, 
the existence of this integral as an element of $\mathcal L(L^2(\mathbb R),L^2(\mathbb R))$
was established for various functionals.
Next, the existence of the integral  as an element of $\mathcal L(L^1(\mathbb R),L^{\infty}(\mathbb R))$  
was  established  in \cite{CS73-2,CS73-3,CPS90,JS75-2}.
Finally,  the $L_p \to L_{p'}$ theory 
($1< p \le 2$ and $1/p+1/p'=1$)  
was developed   as an element of 
$\mathcal L(L^p(\mathbb R), L^{p'}(\mathbb R))$ in  \cite{JS76}. 
 
The Wiener space $C_0[0,T]$ can be considered as the 
space of sample paths of standard Brownian motion process (SBMP).
Thus, in various Feynman integration theories,  
the integrand $F$ of the   Feynman integral \eqref{eq:intro}
is a functional of the SBMP,
see \cite{CS68,CS70,CS73-1,CS73-2,CS73-3,CPS90,JS70-1,JS70-2,JS71,JS73-1,JS73-2,JS74-1,JS74-2,JS75-1,JS75-2,JS76}.

Let $D=[0,T]$ and 
let $(\Omega, {\mathcal F}, P)$ be a probability space.
By the definition, a generalized Brownian motion process (GBMP) on $D\times \Omega$ 
is a Gaussian  process $Y \equiv\{Y_t\}_{t\in D}$
such that $Y_0=0$ almost  surely
and for any $0\leq s< t\leq T$,
\[%%begin{equation}
 Y_t -Y_s \sim N \big(a(t)-a(s), \,b(t)-b(s)\big),
\]%%end{equation} 
where  $N(m, \sigma^2)$ denotes the normal distribution
with mean $m$ and variance $\sigma^2$, 
$a(t)$ is a continuous real-valued function on  $[0,T]$
and $b(t)$ is a increasing  continuous  real-valued function on $[0,T]$.
Thus a GBMP is determined by the continuous functions $a(t)$ and $b(t)$.
The function space $C_{a,b}[0,T]$, induced by GBMP, was introduced by 
Yeh \cite{Yeh71,Yeh73} and was used extensively in 
\cite{CC16,CCK15,CCL09,CCS03,CS03,CCC13,CS20}.
The function space $C_{a,b}[0,T]$ used in
\cite{CC16,CCK15,CCL09,CCS03,CS03,CCC13,CS20}
can be considered as the space of sample paths of the GBMP.
The generalized Feynman integral studied in 
\cite{CC16,CCK15,CCS03,CS03}
are scalar-valued. 
In this paper, \textit{the analytic operator-valued    generalized Feynman integral} (AOVG`Feynman'I) of functionals $F$ on the general Wiener  space $C_{a,b}[0,T]$ is 
investigated as an  element of $\mathcal L(L^1(\mathbb R,\nu_{\delta,a}),  L^{\infty}(\mathbb R))$,
where $\nu_{\delta,a}$ is a  measure on $\mathbb R$
given by   
\[%%begin{equation}
d\nu_{\delta,a}=\exp\{\delta \mathrm{Var}(a)u^2\} du,
\]%%end{equation} 
and where  $\delta>0$ and $\mathrm{Var}(a)$ denotes the total variation of the
mean function $a$ of the  GBMP.
It turns out in this paper that
the AOVG`Feynman'Is
of functionals defined by the stochastic Fourier--Stieltjes transform 
of complex measures on the infinite dimensional Hilbert space 
$C_{a,b}'[0,T]$, the space of absolutely continuous functions in $C_{a,b}[0,T]$, 
are elements of the linear space 
\[
\bigcap_{\delta>0} \mathcal L( L^1(\mathbb R,\nu_{\delta,a}),L^{\infty}(\mathbb R)).
\]

Note that when $a(t)\equiv 0$ and $b(t)=t$, 
the GBMP is an  SBMP, and so 
the function space $C_{a,b}[0,T]$ reduces to 
the classical Wiener space $C_0[0,T]$.
But we are obliged to point out that an SBMP  used in
\cite{CS68,CS70,CS73-1,CS73-2,CS73-3,CPS90,JS70-1,JS70-2,JS71,JS73-1,JS73-2,JS74-1,JS74-2,JS75-1,JS75-2,JS76} 
is stationary in time  and is free of drift.
While, the  GBMP used in this paper, as well as 
in \cite{CC16,CCK15,CCL09,CCS03,CS03,CCC13,CS20},
is nonstationary in time and is subject to a  drift $a(t)$.
It turns out, as noted in Remark \ref{re:point} below,
that including a drift term $a(t)$ makes establishing the existence  of 
\textit{the analytic operator-valued generalized  function space integral} (AOVGFSI)
and  AOVG`Feynman'I  of functionals on $C_{a,b} [0,T]$ very difficult.

\par
The results in this paper are quite a lot 
more complicated because the GBMP is neither stationary nor centered.
 
%%%%%%%%%%%%%%%%%%%%%%%%%%%%%%%%%%%%%%%%%%%%%%%%%%%%%%%%%%%%%%%%%%
%%%%%%%%%%%%%%%%%%%%%%%%%%%%%%%%%%%%%%%%%%%%%%%%%%%%%%%%%%%%%%%%%%
%%%-----------------------[section]----------------------------%%%
%%%%%%%%%%%%%%%%%%%%%%%%%%%%%%%%%%%%%%%%%%%%%%%%%%%%%%%%%%%%%%%%%%
%%%%%%%%%%%%%%%%%%%%%%%%%%%%%%%%%%%%%%%%%%%%%%%%%%%%%%%%%%%%%%%%%%
\setcounter{equation}{0}
\section{Preliminaries}\label{sec:2}

In this section,
we briefly list some of the preliminaries from \cite{CC16,CCK15,CCS03,CS03} 
that we need to establish our results in next sections;
for more details, see  \cite{CC16,CCK15,CCS03,CS03}. 

\par
 \par 
Let $(C_{a,b}[0,T],\mathcal B(C_{a,b}[0,T]),\mu )$ 
denote  the function space induced by the GBMP $Y$ determined by 
continuous functions  $a(t)$ and $b(t)$,  
where $\mathcal B(C_{a,b}[0,T])$ is the Borel $\sigma$-field
induced by the $\sup$-norm, see \cite{Yeh71,Yeh73}.
We assume in this paper that 
$a(t)$ is an absolutely continuous real-valued  function 
on $[0,T]$   with $a(0)=0$, $a'(t) \in L^2[0,T]$,  
and $b(t)$ is an increasing, 
continuously differentiable real-valued function
with $b(0)=0 $ and $b'(t) >0$ for each $t \in [0,T]$.
We  complete this function space
to obtain  the complete probability  measure space $(C_{a,b}[0,T],\mathcal W(C_{a,b}[0,T]),\mu)$
where $\mathcal W(C_{a,b}[0,T])$ is the set of all $\mu$-Carath\'eodory 
measurable subsets of $C_{a,b}[0,T]$.

We can consider the coordinate process
$X:[0,T] \times C_{a,b}[0,T]\to\mathbb R$ given by $X(t,x)=x(t)$
which is a continuous process.
The   separable process $X$ induced by $Y$ \cite{Yeh73}
also has the following properties:
\begin{itemize}
%-----
\item[(i)]
$X(0,x)=x(0)=0$ for every $x \in C_{a,b}[0,T]$. 
%-----
\item[(ii)]  
For any $s,t\in [0,T]$ with $s\le t$ and $x\in C_{a,b}[0,T]$,  
\[%%begin{equation}
 x(t)- x(s)\sim N(a(t)-a(s), b(t)-b(s)).
\]%%end{equation}
\end{itemize}
%-----
Thus it follows that  for $s, t\in [0,T]$, 
$\mathrm{Cov}(X(s,x),X(t,x)) =\min\{b(s),b(t)\}$.

\par 
A subset $B$ of $C_{a,b}[0,T]$ is said to be scale-invariant 
measurable  provided $\rho B$
is $\mathcal W(C_{a,b}[0,T])$-measurable for all   $\rho>0$, and 
a scale-invariant measurable set $N$ is said to be   scale-invariant null 
provided   $\mu (\rho N)= 0$  for all   $\rho > 0$.
A property that holds except on a scale-invariant null set is said to hold 
scale-invariant almost everywhere (s-a.e.).
A functional $F$ is said to be scale-invariant measurable 
provided $F$ is defined on a scale-invariant measurable set 
and $F(\rho\,\,\cdot\,)$ 
is ${\mathcal{W}}(C_{a,b}[0,T])$-measurable for every $\rho>0$.

\par
Let $L_{a,b}^2[0,T]$ be the separable Hilbert space of functions on $[0,T]$ 
which are  Lebesgue measurable and square integrable 
with respect to the Lebesgue--Stieltjes  measures on $[0,T]$ 
induced by      $b(t)$ and  $a(t)$: i.e.,
\[%%begin{equation}
L_{a,b}^2[0,T] 
=\bigg\{v :  \int_{0}^{T} v^2 (s) db(s)  <+\infty  \,\hbox{ and }\,  
             \int_{0}^{T} v^2 (s) d|a|(s)  <+\infty \bigg\}      			 
\]%%end{equation}
where $|a|(t)$ denotes the total variation function of $a(t)$ on $[0,T]$.
The inner product on $L_{a,b}^2[0,T]$ is defined by
$(u,v)_{a,b}=\int_0^T u(t)v(t)d[ b(t) +|a|(t)]$.
Note that $\| u\|_{a,b} = \sqrt{(u,u)_{a,b}}=0$ 
if and only if $u(t)=0$ a.e. on $[0,T]$ and
that all functions of bounded variation on $[0,T]$
are elements of $L_{a,b}^2[0,T]$.
Also note that if $a(t)\equiv 0$ and $b(t) = t$, 
then $L_{a,b}^2[0,T]=L^2[0,T]$.
In fact,
\[%%begin{equation}
(L_{a,b}^2[0,T] ,\|\cdot\|_{a,b})
\subset (L^2_{0,b}[0,T],\|\cdot\|_{0,b})
= (L^2[0,T], \|\cdot\|_2)
\]%%end{equation}
since the two norms $\|\cdot\|_{0,b}$ 
and $\|\cdot\|_2$ are equivalent.

\par
Throughout the rest of this paper, we consider the linear space
\[%%begin{equation}
C_{a,b}'[0,T]
 =\bigg\{ w \in C_{a,b}[0,T] : w(t)=\int_0^t z(s) d b(s)
\hbox{  for some   } z \in L_{a,b}^2[0,T]  \bigg\}.
\]%%end{equation}
For $w\in C_{a,b}'[0,T]$, with $w(t)=\int_0^t z(s) d b(s)$ for $t\in [0,T]$,
let $D: C_{a,b}'[0,T] \to L_{a,b}^2[0,T]$ be defined by the formula
\begin{equation}\label{eq:Dt}
Dw(t)= z(t)=\frac{w'(t)}{b'(t)}.
\end{equation}
Then $C_{a,b}' \equiv C_{a,b}'[0,T]$ with inner product
\[%%begin{equation}
(w_1, w_2)_{C_{a,b}'}
=\int_0^T  Dw_1(t)  Dw_2(t)  d b(t)
=\int_0^T  z_1(t) z_2(t)  d b(t)
\]%%end{equation}
is also a separable  Hilbert space.

\par
Note that  the two separable Hilbert spaces $L_{a,b}^2[0,T]$ and $C_{a,b}'[0,T]$
are  topologically homeomorphic under the linear operator given by 
equation \eqref{eq:Dt}. The inverse operator of $D$ is given by
\[%%begin{equation} 
(D^{-1}z)(t)=\int_0^t z(s) d b(s)
\]%%end{equation}
for $t\in [0,T]$.

\par
In this paper, in addition to the conditions put on $a(t)$ above,
we now add the condition
\begin{equation}\label{eq:new-cc2}
\int_0^T |a'(t)|^2 d|a|(t)< +\infty.
\end{equation}
Then, the function $a: [0,T]\to\mathbb R$ satisfies the condition 
\eqref{eq:new-cc2} if and only if $a(\cdot)$ is an element of $C_{a,b}'[0,T]$, 
see \cite{CCC13,CS20}.  
Under the  condition \eqref{eq:new-cc2}, we observe that for each  
$w\in C_{a,b}'[0,T]$ with $Dw=z$,
\[%%begin{equation}
(w,a)_{C_{a,b}'}=\int_0^T Dw(t) Da(t) db(t)
=\int_0^T z(t)\frac{a'(t)}{b'(t)}db(t)=\int_0^T z(t)da(t).
\]%%end{equation}

Next we will define a Paley--Wiener--Zygmund (PWZ)  stochastic integral.
Let $\{g_j\}_{j=1}^{\infty}$ be a complete orthonormal set in $C_{a,b}'[0,T]$
such that for each $j=1,2,\ldots$, $Dg_j=\alpha_j$ is of bounded variation on $[0,T]$.
For each $w=D^{-1}z\in C_{a,b}'[0,T]$,
the PWZ stochastic integral
$(w,x)^{\sim}$  is defined by the formula
\[%%begin{equation}
\begin{aligned}
(w,x)^{\sim}
&
=\lim_{n\to\infty} \int_{0}^{T}\sum_{j=1}^{n}(w,g_{j})_{C_{a,b}'}Dg_{j}(t)dx(t)\\
&=\lim_{n\to\infty} \int_{0}^{T}\sum_{j=1}^{n}\int_0^T z(s) \alpha_{j}(s)db(s)\alpha_{j}(t)dx(t)\\
\end{aligned}
\]%%end{equation}
for all $x\in C_{a,b}[0,T]$  for which the limit exists.

\par 
It is known that  for   each $w\in C_{a,b}'[0,T]$, 
the PWZ stochastic integral $(w,x)^{\sim}$  exists  for $\mu$-a.e. $x\in C_{a,b}[0,T]$. 
If $Dw=z\in L_{a,b}^2[0,T]$ is of bounded variation on $[0,T]$, then the PWZ 
stochastic integral $(w,x)^{\sim}$ equals the Riemann--Stieltjes integral 
$\int_0^T z(t)dx(t)$.   It also  follows  that for $w, x\in C_{a,b}'[0,T]$, 
$(w,x)^{\sim}=(w,x)_{C_{a,b}'}$.
%----------------   Remark    -----------------%
For each $w\in C_{a,b}'[0,T]$,  the PWZ stochastic integral 
$(w,x)^{\sim}$ is a Gaussian random variable on $C_{a,b}[0,T]$   with mean $(w,a)_{C_{a,b}'}$ 
and variance $\|w\|_{C_{a,b}'}^2$.
Note that for all $w_1,w_2  \in C_{a,b}'[0,T]$,
%-----Eq
\[%%begin{equation}
\int_{C_{a,b}[0,T]} 
  (w_1,x)^{\sim}(w_2,x)^{\sim}   d \mu(x)
= (w_1,w_2)_{C_{a,b}'}
 + (w_1,x)_{C_{a,b}'}(w_2,x)_{C_{a,b}'}.
\]%%end{equation}
%-----
Hence we see that for  $w_1,w_2 \in C_{a,b}'[0,T]$, 
$(w_1,w_2)_{C_{a,b}'}=0$ if and only if $(w_1,x)^{\sim}$  and $(w_2,x)^{\sim}$ 
are independent random variables.
We thus have the following function space integration formula:
let $\{e_1,\ldots,e_n\}$ be an orthonormal set in 
$(C_{a,b}'[0,T],\|\cdot\|_{C_{a,b}'})$,
and given a Lebesgue  measurable function $r:\mathbb R^n \to \mathbb C$,
 let $R:C_{a,b}[0,T]\to\mathbb C$ be given by equation 
%-----Eq
\[%%begin{equation}
R(x)= r((e_1,x)^{\sim},\ldots,(e_n,x)^{\sim}).
\]%%end{equation}
%-----
Then  
%-----Eq
\begin{equation}\label{eq:c-formula}
\begin{aligned}
&\int_{C_{a,b}[0,T]}  R(x) d\mu(x)
 \equiv\int_{C_{a,b}[0,T]} r((e_1,x)^{\sim},\ldots,(e_n,x)^{\sim})d\mu(x)\\   
& =(2\pi)^{-n/2}  \int_{\mathbb R^n} r(u_1,\ldots,u_n)
\exp \bigg\{-\sum_{j=1}^{n}\frac{(u_j-(e_j,a)_{C_{a,b}'})^2 }{2}\bigg\}
     du_1\cdots du_n  
\end{aligned}
\end{equation}
%-----
in the sense that if either side of equation \eqref{eq:c-formula} exists, 
both sides exist  and equality holds.

\par
The following integration formula is also  used   in this paper:
\begin{equation}\label{eq:int-formula}
\int_\mathbb{R} \exp \{-\alpha u^2+ \beta u \} du 
= \sqrt{\frac{\pi}{\alpha}} \exp \Big\{ \frac{\beta^2}{4\alpha}  \Big\}
\end{equation}
for  complex numbers $\alpha $ and $\beta$ with $\hbox{Re}(\alpha)> 0$.

%%%%%%%%%%%%%%%%%%%%%%%%%%%%%%%%%%%%%%%%%%%%%%%%%%%%%%%%%%%%%%%%%%
%%%%%%%%%%%%%%%%%%%%%%%%%%%%%%%%%%%%%%%%%%%%%%%%%%%%%%%%%%%%%%%%%%
%%%-----------------------[section]----------------------------%%%
%%%%%%%%%%%%%%%%%%%%%%%%%%%%%%%%%%%%%%%%%%%%%%%%%%%%%%%%%%%%%%%%%%
%%%%%%%%%%%%%%%%%%%%%%%%%%%%%%%%%%%%%%%%%%%%%%%%%%%%%%%%%%%%%%%%%%
\setcounter{equation}{0}
\section{Analytic operator-valued generalized function space  integral}

In this section, we introduce the definition of 
the AOVGFSI as an element of 
$\mathcal{L}(L^1(\mathbb R),L^{\infty}(\mathbb R))$.
The definition below is based on the previous definitions in 
\cite{CS73-1,CS73-2,CS73-3,CPS90,JS75-1,JS75-2,JS76}.

%%%%%%%%%%-----------------------------------------------------------%
%%%%\renewcommand{\thedefinition}{\thesection.1}
%%%%%%%%%%-----------------------------------------------------------%  
%%%%%%%%%%-----------------------------------------------------------%
\begin{definition}\label{def:op-basic2}
Let $F: C[0,T] \to \mathbb C$ be a  measurable functional
and let $h$ be an element of $C_{a,b}'[0,T]\backslash\{0\}$. 
Given $\lambda>0$, $\psi\in L^1(\mathbb R)$ and $\xi \in \mathbb R$,
let 
\begin{equation}\label{eq:op-basic2}
(I_{\lambda}(F;h)\psi)(\xi)
\equiv \int_{C_{a,b}[0,T]} 
   F(\lambda^{-1/2} x  +\xi)
   \psi(\lambda^{-1/2} (h, x)^{\sim} +\xi) d \mu(x). 
\end{equation}
If $I_{\lambda}(F;h)\psi$ is in $L^{\infty}(\mathbb R)$ 
as a function of $\xi$
and if the correspondence $\psi \to I_{\lambda}(F;h)\psi$
gives an element of $\mathcal{L}(L^1(\mathbb R),L^{\infty}(\mathbb R))$,
we say that the operator-valued generalized function space integral (OVGFSI)
$I_{\lambda}(F;h)$ exists.

Let $\Gamma$  be a  region in  $\mathbb C_+$ such that 
$\hbox{\rm Int}(\Gamma)$ is a simply connected domain in $\mathbb C_+$ and
$\hbox{\rm Int}(\Gamma) \cap (0,+\infty)$ is a nonempty open interval
of positive real numbers.   
Suppose that  there exists an $\mathcal{L}(L^1(\mathbb R),L^{\infty}(\mathbb R))$-valued function
which is analytic in $\lambda$ on $\hbox{\rm Int}(\Gamma)$ and agrees with 
$I_{\lambda}(F;h)$ on $\hbox{\rm Int}(\Gamma) \cap (0,+\infty)$, %%$(\alpha,\beta)$,
then this $\mathcal{L}(L^1(\mathbb R),L^{\infty}(\mathbb R))$-valued function is denoted by $I_{\lambda}^{\mathrm{an}}(F;h)$
and is called the   AOVGFSI of $F$ associated with $\lambda$. 
\end{definition}

The notation $\|\cdot\|_{\rm o}$
will be used for the norm of operators in $\mathcal{L}(L^1(\mathbb R),L^{\infty}(\mathbb R))$.
 
%%%%%%%%%%-----------------------------------------------------------%
%%%%\renewcommand{\theremark}{\thesection.2}
%%%%%%%%%%-----------------------------------------------------------%  
%%%%%%%%%%-----------------------------------------------------------%
%%%%%%%%%%%   Remark    %%%%%%%%%%%%%%%%%%%%%
\begin{remark}
{\rm (i)}
In equation  \eqref{eq:op-basic2} above,
choosing $h(t)=\int_0^t  d  b(s)=b(t)\in C_{a,b}'[0,T]$, we obtain
\[%%begin{equation}
(h,x)^{\sim}=(b,x)^{\sim}=\int_0^T Db(t) dx(t)=\int_0^T  dx(t)=x(T).
\]%%end{equation}
In this case, equation \eqref{eq:op-basic2} is given by
\begin{equation}\label{eq:op-oldstyle2}
\begin{aligned}
(I_{\lambda}&(F;b)\psi)(\xi)
=  \int_{C_{a,b}[0,T]} F(\lambda^{-1/2} x+\xi)\psi(\lambda^{-1/2} x(T)+\xi ) d \mu(x).%%\tag2.12
\end{aligned}
\end{equation}
 Moreover,  if $a(t)\equiv 0$ and $b(t)=t$ on $[0,T]$,
then  the function space $C_{a,b}[0,T]$ reduces to the classical Wiener space $C_0[0,T]$
and the definition of the  OVGFSI
$I_{\lambda}(F;b)$ in equation \eqref{eq:op-oldstyle2} agrees with the  definitions 
of the  operator-valued function space integrals $I_{\lambda}(F)$ with $\lambda>0$
defined in  
\cite{CS68,CS70,CS73-1,CS73-2,CS73-3,CPS90,JS70-1,JS70-2,JS71,JS73-1,JS73-2,JS74-1,JS74-2,JS75-1,JS75-2,JS76}.

{\rm (ii)}  In   the  case  that
$a(t)\equiv 0$ and $h(t)=b(t)=t$ on $[0,T]$,
choosing $\Gamma=\mathbb C_+ \cap \{\lambda\in \mathbb C: |\lambda|<\lambda_0\}$ for some
$\lambda_0\in (0, +\infty)$, then the definition of the AOVGFSI
 $I_{\lambda}^{\mathrm{an}}(F;b)$  (if it exists)
agrees with the  definitions of the analytic 
operator-valued function space integral  $I_{\lambda}^{\mathrm{an}}(F)$ associated
with $\lambda>0$ defined in \cite{JS75-2,JS76}.
\end{remark}
%%%%%%%%%%%%%%%%%%%%%%%%%%%%%%%%%%%%%%%%%%%%%

%%%%%%%%%%%%%%%%%%%%%%%%%%%%%%%%%%%%%%%%%%%%%%%%%%%%%%%%%%%%%%%%%%
%%%%%%%%%%%%%%%%%%%%%%%%%%%%%%%%%%%%%%%%%%%%%%%%%%%%%%%%%%%%%%%%%%
%%%-----------------------[section]----------------------------%%%
%%%%%%%%%%%%%%%%%%%%%%%%%%%%%%%%%%%%%%%%%%%%%%%%%%%%%%%%%%%%%%%%%%
%%%%%%%%%%%%%%%%%%%%%%%%%%%%%%%%%%%%%%%%%%%%%%%%%%%%%%%%%%%%%%%%%%
\setcounter{equation}{0}
\section{The $\mathcal{F}(C_{a,b}[0,T])$ theory}\label{sec:fresnel}

In \cite{CC16,CCL09}, Chang, Choi and Lee introduced  a Banach algebra 
$\mathcal{F}(C_{a,b}[0,T])$ of functionals on function space $C_{a,b}[0,T]$,
each of which is  a stochastic Fourier transform
of $\mathbb C$-valued Borel measure on $C_{a,b}'[0,T]$, and showed that 
it contains many   functionals 
of interest in Feynman integration theory.
In \cite{CC16}, Chang and Choi showed that the analytic (but scalar-valued)
generalized Feynman integral 
exists for functionals in $\mathcal{F}(C_{a,b}[0,T])$.
In this section, we show that the AOVGFSI 
$I_{\lambda}^{\mathrm{an}}(F;h)$ is in $\mathcal{L}(L^1(\mathbb R),L^{\infty}(\mathbb R))$
for functionals $F$ in $\mathcal{F}(C_{a,b}[0,T])$.

Let  $\mathcal{M}(C_{a,b}'[0,T])$  denote the
space  of  $\mathbb C$-valued, countably additive (and hence finite)  
Borel measures on $C_{a,b}'[0,T]$. 
We define the Fresnel type class $\mathcal F(C_{a,b}[0,T])$ 
of functionals on $C_{a,b}[0,T]$  as the space 
of all stochastic Fourier--Stieltjes transforms of elements 
of $\mathcal{M}(C_{a,b}'[0,T])$; that is, 
$F \in \mathcal F (C_{a,b}[0,T])$ 
if and only if 
there exists a measure $f$ in $\mathcal M(C_{a,b}'[0,T])$
such that 
%-----Eq
\begin{equation}\label{eq:element}
F(x) =\int_{C_{a,b}'[0,T]}\exp\{i(w,x)^{\sim}\} d f(w)                             
\end{equation}
%-----
for s-a.e. $x\in C_{a,b}[0,T]$.

More precisely, since we shall identify
functionals which coincide s-a.e. on $C_{a,b}[0,T]$,
$\mathcal{F}(C_{a,b} [0,T])$  can be   regarded 
as the space of all s-equivalence classes of 
functionals having the form \eqref{eq:element}.

We note that $\mathcal{M}(C_{a,b}'[0,T])$ is a Banach algebra 
under the total variation norm and with 
convolution as multiplication.
The  Fresnel type class $\mathcal F(C_{a,b}[0,T])$ 
is a Banach algebra  with norm
%-----Eq
\[%%begin{equation}
\|F\|=\|f\|=\int_{C_{a,b}'[0,T]}d|f|(w).
\]%%end{equation}
%-----
In fact, the correspondence $f \mapsto F$
is injective, carries convolution into pointwise multiplication 
and is a Banach algebra isomorphism 
where $f$ and $F$ are related by \eqref{eq:element}. 
For a more detailed study of functionals in $\mathcal{F}(C_{a,b}[0,T])$,
see \cite{CC16,CCL09}.

%%%%%%%%%%-----------------------------------------------------------%
%%%%\renewcommand{\theremark}{\thesection.1}
%%%%%%%%%%-----------------------------------------------------------%  
%%%%%%%%%%-----------------------------------------------------------%
\begin{remark} 
 If $F$ is in $\mathcal F (C_{a,b}[0,T])$, 
then $F$ is scale-invariant measurable 
and s-a.e. defined on $C_{a,b}[0,T]$.
If $x$ in $C_{a,b}[0,T]$ is such that $F(x)$ is defined,  
then by \eqref{eq:element} and the definition of the 
PWZ stochastic integral, it follows that 
$F(x+\xi)=F(x)$ for all $\xi\in \mathbb R$.
\end{remark}

Let $h$ be a (fixed) function in $C_{a,b}'[0,T]\backslash\{0\}$.
Then for any function $w$ in $C_{a,b}'[0,T]$, 
we obtain an orthonormal set $\{e_1,e_2(w)\}$ in $C_{a,b}'[0,T]$,
by  the Gram--Schmidt process, 
such that $h= \|h\|_{C_{a,b}'}e_1$
and  
\begin{equation}\label{eq:GS-expression}
w=(w,e_1)_{C_{a,b}'}e_1+\beta_we_2(w)
\end{equation}
where
\[%%begin{equation}
\beta_w
=\big\|w-(w,e_1)_{C_{a,b}'}e_1\big\|_{C_{a,b}'}
=\Big[\|w\|_{C_{a,b}'}^2 -(w,e_1)_{C_{a,b}'}^2\Big]^{1/2}.
\]%%end{equation}
Throughout this paper, we will use the following notations for convenience:
\begin{equation}\label{eq:simple-K}
M(\lambda; h)=\bigg(\frac{\lambda}{2\pi\|h\|_{C_{a,b}'}^2}\bigg)^{1/2},
\end{equation}
\begin{equation}\label{eq:simple-V}
\begin{aligned}
&V(\lambda; \xi,v\,; h,w)\\
&=\exp\bigg\{ \frac{1}{2\lambda\|h\|_{C_{a,b}'}^2}
   \Big[\big(i\lambda(v-\xi)+(h,w)_{C_{a,b}'}\big)^2
   -\|h\|_{C_{a,b}'}^2\|w\|_{C_{a,b}'}^2\Big]\bigg\},
\end{aligned}
\end{equation}
\begin{equation}\label{eq:simple-L}
L(\lambda; \xi,v\,;h)
=\exp\bigg\{\frac{\lambda}{2} \frac{(v-\xi)^2 }{\|h\|_{C_{a,b}'}^2} \bigg\},
\end{equation}
\begin{equation}\label{eq:simple-H}
H(\lambda; \xi,v\,;h)
=\exp\bigg\{
 -\frac{\big(\sqrt{\lambda}(v-\xi) - (h,a)_{C_{a,b}'}\big)^2}{2\|h\|_{C_{a,b}'}^2} \bigg\},
\end{equation}
\begin{equation}\label{eq:simple-A}
\begin{aligned}
A(\lambda; w)
&=\exp\bigg\{\frac{i}{\sqrt{\lambda}}\beta_w (e_2(w),a)_{C_{a,b}'}\bigg\}\\
&=\exp\bigg\{
 \frac{i}{\sqrt{\lambda}}\Big[\|w\|_{C_{a,b}'}^2 -(w,e_1)_{C_{a,b}'}^2\Big]^{1/2} (e_2(w),a)_{C_{a,b}'}\bigg\}\\
\end{aligned}
\end{equation}
and
\begin{equation}\label{eq:K-lambda}
\begin{aligned}
&(K_{\lambda} (F;h)\psi)(\xi)\\
&= M(\lambda;h)\int_{C_{a,b}'[0,T]}\int_{\mathbb R}
\psi(v)
V(\lambda; \xi,v\,; h,w)\\
& \qquad\qquad\qquad\qquad \times
L(\lambda; \xi,v\,;h)
H(\lambda; \xi,v\,;h)
A(\lambda; w) dv df(w)
\end{aligned}
\end{equation}
for  
$(\lambda, \xi,v,h,w,\psi)\in \widetilde{\mathbb{C}}_+
\times \mathbb R^2  \times (C_{a,b}'[0,T]\backslash\{0\})\times C_{a,b}'[0,T]\times L^1(\mathbb R)$.
 In equation \eqref{eq:simple-A} above, $w$, $e_1$ and
$e_2$ are related by equation \eqref{eq:GS-expression}.

%%%%%%%%%%-----------------------------------------------------------%
%%%%\renewcommand{\theremark}{\thesection.2}
%%%%%%%%%%-----------------------------------------------------------%  
%%%%%%%%%%-----------------------------------------------------------%
\begin{remark}\label{re:point}
Clearly, for $\lambda>0$,  $|H(\lambda; \xi,v\,;h)|\le 1$
for all $(\xi,v,h)\in  \mathbb R^2\times (C_{a,b}'[0,T]\backslash\{0\})$.
But for $\lambda \in \widetilde{\mathbb{C}}_+$, $|H(\lambda; \xi,v ;h)|$
is not necessarily bounded by $1$. Note that for each $\lambda\in \widetilde{\mathbb{C}}_+$,
$\mathrm{Re}(\lambda) \ge0$ and 
$\mathrm{Re}(\sqrt\lambda)\ge|\mathrm{Im}(\sqrt\lambda)|\ge0$.
Hence for each $\lambda \in \widetilde{\mathbb{C}}_+$,
\begin{equation}\label{eq:anal-H}
\begin{aligned}
H(\lambda; \xi,v\,;h)
&=\exp\bigg\{
 -\frac{ [\mathrm{Re}(\lambda)+i\mathrm{Im}(\lambda)](v-\xi)^2}{2\|h\|_{C_{a,b}'}^2} \\
&  \qquad 
 +\frac{[\mathrm{Re}(\sqrt\lambda)+i\mathrm{Im}(\sqrt\lambda)]
   (v-\xi)(h,a)_{C_{a,b}'}}{\|h\|_{C_{a,b}'}^2} 
-\frac{ (h,a)_{C_{a,b}'}^2 }{2\|h\|_{C_{a,b}'}^2} \bigg\},
\end{aligned}
\end{equation}
and so
\begin{equation}\label{eq:|H|}
\begin{aligned}
&\big|H(\lambda; \xi,v\,;h)\big| \\
&=\exp\bigg\{
 -\frac{  \mathrm{Re}(\lambda) (v-\xi)^2}{2\|h\|_{C_{a,b}'}^2} 
 +\frac{\mathrm{Re}(\sqrt\lambda) 
   (v-\xi)(h,a)_{C_{a,b}'}}{\|h\|_{C_{a,b}'}^2} 
-\frac{ (h,a)_{C_{a,b}'}^2 }{2\|h\|_{C_{a,b}'}^2} \bigg\}.
\end{aligned}
\end{equation}

Note that for $\lambda\in \mathbb C_+$,
the case we consider throughout Section \ref{sec:fresnel},
$\mathrm{Re}(\sqrt\lambda)> |\mathrm{Im}(\sqrt\lambda)|\ge0$,
which implies that 
$\mathrm{Re}(\lambda)= [\mathrm{Re}
(\sqrt\lambda)]^2-[\mathrm{Im}(\sqrt\lambda)]^2>0$.
Hence for each $\lambda \in \mathbb C_+$,
$0<|\mathrm{Arg}(\lambda)|  < \pi/2$
and so 
\begin{equation}\label{eq:secant1}
\frac{[\mathrm{Re}(\sqrt\lambda)]^2}{\mathrm{Re}(\lambda)}
=\frac12\bigg(\frac{|\lambda | }{ \mathrm{Re}(\lambda)}+1\bigg)
= \frac12(\sec \mathrm{Arg}(\lambda)+1).
\end{equation}
For $(\lambda,h)\in \mathbb C_+\times  (C_{a,b}'[0,T]\backslash\{0\})$, let
\begin{equation}\label{eq:S}
S(\lambda;h)
=\exp\bigg\{ (\sec \mathrm{Arg}(\lambda)+1) \frac{(h,a)_{C_{a,b}'}^2}{ 4\|h\|_{C_{a,b}'}^2}\bigg\}.
\end{equation}
Using  \eqref{eq:|H|}, \eqref{eq:secant1}, and \eqref{eq:S}, 
we obtain that for all $\lambda \in \mathbb C_+$,
\begin{equation}\label{eq:add-500}
\begin{aligned}
&\big|H(\lambda; \xi,v\,;h)\big| \\
&=\exp\bigg\{
 -\frac{  \mathrm{Re}(\lambda) (v-\xi)^2}{2\|h\|_{C_{a,b}'}^2} 
 +\frac{\mathrm{Re}(\sqrt\lambda) 
   (v-\xi)(h,a)_{C_{a,b}'}}{\|h\|_{C_{a,b}'}^2} 
-\frac{ (h,a)_{C_{a,b}'}^2 }{2\|h\|_{C_{a,b}'}^2} \bigg\}\\
&=\exp\bigg\{
 -\frac{  \mathrm{Re}(\lambda)}{2\|h\|_{C_{a,b}'}^2} 
\bigg[  (v-\xi)-\frac{\mathrm{Re}(\sqrt\lambda)}{\mathrm{Re}(\lambda)}(h,a)_{C_{a,b}'}\bigg]^2\\
&\qquad\quad\,\,\,\,
+\frac{[\mathrm{Re}(\sqrt\lambda)]^2}{\mathrm{Re}(\lambda)}\frac{(h,a)_{C_{a,b}'}^2}{2\|h\|_{C_{a,b}'}^2}
-\frac{ (h,a)_{C_{a,b}'}^2 }{2\|h\|_{C_{a,b}'}^2} \bigg\}\\
&\le S(\lambda;h).
\end{aligned}
\end{equation}
These observations are critical to the development of the existence of
the AOVGFSI   $I_{\lambda}^{\mathrm{an}}(F;h)$.

 One can see that for all 
$(\lambda, \xi,v,h,w)\in  \mathbb{C}_+
\times \mathbb R^2 \times (C_{a,b}'[0,T]\backslash\{0\})\times C_{a,b}'[0,T]$,
\begin{equation}\label{eq:2020-a1}
\begin{aligned}
& \big|V(\lambda; \xi,v\,; h,w)L(\lambda; \xi,v\,;h) \big|\\
&=\Bigg|\exp\Bigg\{ \frac{
\big[\big(i\lambda(v-\xi)+(h,w)_{C_{a,b}'}\big)^2
   -\|h\|_{C_{a,b}'}^2\|w\|_{C_{a,b}'}^2\big]}{2\lambda\|h\|_{C_{a,b}'}^2} 
 +\frac{\lambda}{2}\bigg(\frac{v-\xi }{\|h\|_{C_{a,b}'}}\bigg)^2   \Bigg\}\Bigg|\\
&= \exp\bigg\{  
 -\frac{\mathrm{Re}(\lambda)}{2|\lambda|^2\|h\|_{C_{a,b}'}^2}
\Big[ \|h\|_{C_{a,b}'}^2\|w\|_{C_{a,b}'}^2-(h,w)_{C_{a,b}'}^2 \Big]  \bigg\}\\
&\le1,
\end{aligned}
\end{equation} 
because $(h,w)_{C_{a,b}'}^2\le \|h\|_{C_{a,b}'}^2\|w\|_{C_{a,b}'}^2$.
However,   the expression \eqref{eq:simple-A}
is an unbounded function of $w$ for $w\in C_{a,b}'[0,T]$,
because $\beta_w(e_2(w),a)_{C_{a,b}'}$ with
\begin{equation}\label{e2-GS}
e_2(w)
 = \frac{1}{\beta_w}\big[w-(w,e_1)_{C_{a,b}'}e_1\big]   
 = \frac{1}{\beta_w}\bigg[w-\frac{1}{\|h\|_{C_{a,b}'}^2}(h, w)_{C_{a,b}'}h\bigg]  
\end{equation}
is an unbounded function of $w$ for $w\in C_{a,b}'[0,T]$.
Throughout this section, we thus will need to put additional restrictions on the complex measure
$f$ corresponding to $F$ in order to obtain the existence of our AOVGFSI
$I_{\lambda}^{\mathrm{an}}(F;h)$ of $F$ in $\mathcal{F}(C_{a,b}[0,T])$.
\end{remark}

\par
In order to obtain the existence of the AOVGFSI,
we  need to impose additional restrictions on the functionals
in $\mathcal{F}(C_{a,b}[0,T])$.

For a positive real number $q_0$, let
\begin{equation} \label{eq;kq0w}
k(q_0;w)=\exp\big\{(2q_0)^{-1/2}\|w\|_{C_{a,b}'}\|a\|_{C_{a,b}'}\big\}
\end{equation}
and let
\begin{equation}\label{eq:domain}
%%%\begin{aligned}
\Gamma_{q_0}
%%%&
= \bigg\{\lambda\in \widetilde{\mathbb{C}}_+ : |\mathrm{Im}(\lambda^{-1/2})|
=\sqrt{\tfrac{|\lambda|-\mathrm{Re}(\lambda)}{2|\lambda|^2}}   < (2q_0)^{-1/2} \bigg\}.
%%%
%%%\\
%%%&= \bigg\{\lambda\in \widetilde{\mathbb{C}}_+ :
%%%\sqrt{\tfrac{|\lambda|-\mathrm{Re}(\lambda)}{2|\lambda|^2}}  < (2q_0)^{-1/2} \bigg\}.
%%%
%%%\end{aligned}
\end{equation}
Define 
a subclass $\mathcal{F}^{q_0}$
of $ \mathcal{F}(C_{a,b}[0,T])$ by
$F\in \mathcal{F}^{q_0} $ 
if and only if
\begin{equation}\label{eq:condition-finite}
\int_{C_{a,b}'[0,T]}k(q_0;w)d |f|(w)
<+\infty.
\end{equation}
Then for all  $\lambda \in \Gamma_{q_0}$, 
\begin{equation}\label{2020-a2}
|A(\lambda;w)|< k(q_0;w).
\end{equation}

%%%%%%%%%%-----------------------------------------------------------%
%%%%\renewcommand{\theremark}{\thesection.3}
%%%%%%%%%%-----------------------------------------------------------%  
%%%%%%%%%%-----------------------------------------------------------%
\begin{remark} 
 The region $\Gamma_{q_0}$ given by \eqref{eq:domain} satisfies the conditions
stated in Definition \ref{def:op-basic2}; i.e.,
$\mathrm{Int}(\Gamma_{q_0})$ is a simple connected domain in $\mathbb C_+$
and $\mathrm{Int}(\Gamma_{q_0})\cap(0,+\infty)$ is an open interval.
We note that for all real $q$ with $|q|>q_0$,
\[%%begin{equation}
(-iq)^{-1/2}=\frac{1}{\sqrt{2|q|}}+i\frac{\mathrm{sign}(q)}{\sqrt{2|q|}}.
\]%%end{equation}
Also, by a close examination of  \eqref{eq:domain}, 
it follows that $-iq$  
is an element of the region $\Gamma_{q_0}$. 
In fact, $\Gamma_{q_0}$ is a simple connected  
neighborhood of $-iq$ in $\widetilde{\mathbb C}_+$.
\end{remark}

%%%%%%%%%%-----------------------------------------------------------%
%%%%\renewcommand{\thelemma}{\thesection.4}
%%%%%%%%%%-----------------------------------------------------------%  
%%%%%%%%%%-----------------------------------------------------------%
\begin{lemma} \label{lemma;1-infty}
Let $q_0$ be a positive real number  and 
let $F$ be an element of $\mathcal{F}^{q_0}$.
Let $h$ be an element of $C_{a,b}'[0,T]\backslash\{0\}$
and let $\Gamma_{q_0}$ be given by   \eqref{eq:domain}.
Let $(K_{\lambda} (F;h)\psi)(\xi)$ be given by equation 
\eqref{eq:K-lambda} for $(\lambda,\xi,  \psi)
\in \Gamma_{q_0} \times \mathbb R \times L^1(\mathbb R)$.
Then  $K_{\lambda} (F;h)$ 
is an element of $\mathcal{L}(L^1(\mathbb R),L^{\infty}(\mathbb R))$ for each 
$\lambda \in \mathrm{Int}(\Gamma_{q_0})$. 
\end{lemma}
\begin{proof}
Let $\Gamma_{q_0}$ be given by \eqref{eq:domain}.
Using \eqref{eq:K-lambda},  
  \eqref{eq:simple-K}, \eqref{eq:simple-V}, \eqref{eq:simple-L},
\eqref{eq:simple-H},   \eqref{eq:simple-A}, \eqref{eq:2020-a1}, 
the Fubini theorem,    \eqref{eq:add-500}, and \eqref{2020-a2},
we  observe that for all 
 $(\lambda,\xi,  \psi)\in \mathrm{Int}(\Gamma_{q_0})\times \mathbb R \times L^1(\mathbb R)$,
\begin{equation}\label{eq:add-number001}
\begin{aligned}
&\big|(K_{\lambda}  (F;h)\psi)(\xi)\big|\\
& \le
 M(|\lambda|;h)
\int_{C_{a,b}'[0,T]}\int_{\mathbb R} \big|\psi(v)\big|
\big| V(\lambda; \xi,v\,; h,w) L(\lambda; \xi,v\,;h)\big|\\
&\qquad\qquad\qquad\qquad\qquad
\times
\big| H(\lambda; \xi,v\,;h)\big| 
    \big| A(\lambda; w)\big| d v d|f|(w)\\
&\le   M(|\lambda|;h)\int_{\mathbb R}  \big|\psi(v)\big| \big| H(\lambda; \xi,v\,;h)\big| d v
   \int_{C_{a,b}'[0,T]} \big| A(\lambda; w)\big|   d|f|(w)\\
&\le\|\psi \|_1    S(\lambda;h)
  M(|\lambda|;h)  \int_{C_{a,b}'[0,T]}k(q_0;w) d |f|(w)\\
&
<+\infty,
\end{aligned}
\end{equation}
where $S(\lambda;h)$ is given by equation \eqref{eq:S}.
Clearly $K_{\lambda}(F;h): L_1(\mathbb R)\to L_{\infty}(\mathbb R)$ is linear.
Thus, for all $\lambda\in \mathrm{Int}(\Gamma_{q_0})$,  
\[%%begin{equation}
\begin{aligned}
\|K_{\lambda}(F;h)\|_{\rm o}
&\le 
 S(\lambda;h)
  M(|\lambda|;h) \int_{C_{a,b}'[0,T]} k(q_0;w) d |f|(w) 
\end{aligned}
\]%%end{equation}
and the lemma is proved.
  \end{proof}

%%%%%%%%%%-----------------------------------------------------------%
%%%%\renewcommand{\thelemma}{\thesection.5}
%%%%%%%%%%-----------------------------------------------------------%  
%%%%%%%%%%-----------------------------------------------------------%
\begin{lemma} \label{lemma;2-infty}
Let $q_0$, $F$, $h$,   $\Gamma_{q_0}$ and
$(K_{\lambda} (F;h)\psi)(\xi)$ be as in Lemma \ref{lemma;1-infty}.
Then  $(K_{\lambda} (F;h)\psi)(\xi)$  is an analytic 
function of $\lambda$ on $\hbox{\rm Int}(\Gamma_{q_0})$.
\end{lemma}
\begin{proof}
Let $\lambda \in \mathrm{Int}(\Gamma_{q_0})$ be given and let
$\{\lambda_l\}_{l=1}^{\infty}$ be a sequence in $\mathbb C_+$
such that $\lambda_l\to \lambda$.
Clearly,  $0\le |\mathrm{Arg}(\lambda)| < \pi/2$.
Thus there exist $ \theta_0\in (\mathrm{Arg}(\lambda), \pi/2)$ and 
$n_0\in \mathbb N$ such that $\lambda_l \in \mathrm{Int}(\Gamma_{q_0})$
and $0<|\mathrm{Arg}(\lambda_l)|< \theta_0$ for 
all   $l>n_0$. We first note that for each $l>n_0$, 
\[%%begin{equation}
\frac{[\mathrm{Re}(\sqrt{\lambda_l })]^2}{\mathrm{Re}(\lambda_l)}
=\frac12\bigg(\frac{|\lambda_l| }{ \mathrm{Re}(\lambda_l)}+1\bigg)
= \frac12(\sec \mathrm{Arg}(\lambda_l)+1) 
< \frac12(\sec \theta_0+1).
\]%%end{equation}
Using this and the Cauchy--Schwartz inequality,
 it follows  that for all $l>n_0$ and $\psi \in L^{1}(\mathbb R)$,
\begin{equation}\label{eq:add-300}
\begin{aligned}
& \big| \psi(v)\big| \big|V(\lambda_l; \xi, v\,;h,w) L(\lambda_l; \xi, v\,;h)
H(\lambda_l; \xi, v\,;h)
A(\lambda_l;  w)\big|\\
&= \big| \psi(v)\big|
\exp\Bigg\{ -\frac{ \mathrm{Re}(\lambda_l)(v-\xi)^2}{2\|h\|_{C_{a,b}'}^2}\\
&\qquad \qquad\qquad 
 -\frac{\mathrm{Re}(\lambda_l)}{2|\lambda_l|^2\|h\|_{C_{a,b}'}^2}
\Big[ \|h\|_{C_{a,b}'}^2\|w\|_{C_{a,b}'}^2-(h,w)_{C_{a,b}'}^2 \Big]\\
\end{aligned}
\end{equation}
%%%%%%%%%%%%%
%%%%%%%%%%%%%
\[
\begin{aligned}
&\qquad \qquad\qquad 
 +\frac{\mathrm{Re}(\sqrt{\lambda_l})(v-\xi)(h,a)_{C_{a,b}'}}{\|h\|_{C_{a,b}'}^2}
   -\frac{(h,a)_{C_{a,b}'}^2}{2\|h\|_{C_{a,b}'}^2}\\
& \qquad\qquad \qquad
- \mathrm{Im}(\lambda_l^{-1/2})\Big[ \|w\|_{C_{a,b}'}^2-(w,e_1)_{C_{a,b}'}^2 \Big]^{1/2}
(e_2(w),a)_{C_{a,b}'} \Bigg\}\\
&\le
 \big| \psi(v)\big|
\exp\Bigg\{ -\frac{ \mathrm{Re}(\lambda_l)(v-\xi)^2}{2\|h\|_{C_{a,b}'}^2}
 +\frac{\mathrm{Re}(\sqrt{\lambda_l})(v-\xi)(h,a)_{C_{a,b}'}}{\|h\|_{C_{a,b}'}^2}
   -\frac{(h,a)_{C_{a,b}'}^2}{2\|h\|_{C_{a,b}'}^2}\\
& \qquad\qquad \qquad
- \mathrm{Im}(\lambda^{-1/2})\Big[ \|w\|_{C_{a,b}'}^2-(w,e_1)_{C_{a,b}'}^2 \Big]^{1/2}
(e_2(w),a)_{C_{a,b}'} \Bigg\}\\
&= 
\big| \psi(v)\big|
\exp\Bigg\{ -\frac{ \mathrm{Re}(\lambda_l)}{2\|h\|_{C_{a,b}'}^2}\bigg[(v-\xi) 
    - \frac{\mathrm{Re}(\sqrt{\lambda_l})}{\mathrm{Re}(\lambda_l)}(h,a)_{C_{a,b}'} \bigg]^2\\
&\qquad\qquad\qquad
   +\frac{[\mathrm{Re}(\sqrt{\lambda_l})]^2}{2\|h\|_{C_{a,b}'}^2 \mathrm{Re}(\lambda_l)}(h,a)_{C_{a,b}'}^2
  -\frac{(h,a)_{C_{a,b}'}^2}{2\|h\|_{C_{a,b}'}^2}\\
& \qquad\qquad \qquad   
- \mathrm{Im}(\lambda^{-1/2})\Big[ \|w\|_{C_{a,b}'}^2-(w,e_1)_{C_{a,b}'}^2 \Big]^{1/2}(e_2(w),a)_{C_{a,b}'} \Bigg\}\\
&\le
\big| \psi(v)\big|
\exp\Bigg\{\frac{(h,a)_{C_{a,b}'}^2}{2\|h\|_{C_{a,b}'}^2 }\frac{[\mathrm{Re}(\sqrt{\lambda_l})]^2}{ \mathrm{Re}(\lambda_l)} \\
& \qquad\qquad\qquad 
+ \big|\mathrm{Im}(\lambda^{-1/2})\big|\Big[ \|w\|_{C_{a,b}'}^2-(w,e_1)_{C_{a,b}'}^2 \Big]^{1/2}|(e_2(w),a)_{C_{a,b}'}| \Bigg\}\\
&=
\big| \psi(v)\big| 
\exp\Bigg\{\frac{(h,a)_{C_{a,b}'}^2}{2\|h\|_{C_{a,b}'}^2 }\frac{[\mathrm{Re}(\sqrt{\lambda_l})]^2}{ \mathrm{Re}(\lambda_l)} 
+\big|\mathrm{Im}(\lambda^{-1/2})\big| \|w\|_{C_{a,b}'}  \|a\|_{C_{a,b}'} \Bigg\}\\
&<
\big| \psi(v)\big|
\exp\Bigg\{\frac{(h,a)_{C_{a,b}'}^2}{4\|h\|_{C_{a,b}'}^2 } (\sec\theta_0 +1)\Bigg\}
k(q_0;w)
\end{aligned}
\]%%end{equation}
where $e_2(w)$ and $k(q_0;w)$ are given by \eqref{e2-GS} and \eqref{eq;kq0w}, respectively.
Since $\psi \in L^1(\mathbb R)$, and $f$, the corresponding measure of $F$ by \eqref{eq:element}, 
 satisfies condition \eqref{eq:condition-finite},
the last expression of \eqref{eq:add-300} is integrable on
the product space ($\mathbb R\times C_{a,b}'[0,T],\mathrm{m}_L\times f)$,
as a function of $(v,w)$, where $\mathrm{m}_L$ denotes the Lebesgue measure on $\mathbb R$.
Hence by the dominated convergence theorem, we see that
the right-hand side of equation \eqref{eq:K-lambda} is a continuous function of  
$\lambda$ on $\mathrm{Int}(\Gamma_{q_0})$.
Next we note that for all $(\xi, v, h, w) \in \mathbb R^2\times (C_{a,b}'[0,T]\backslash\{0\})\times C_{a,b}'[0,T]$,
\[%%begin{equation}
V(\lambda ; \xi, v\,;h,w) 
L(\lambda ; \xi, v\,;h)
H(\lambda ; \xi, v\,;h)
A(\lambda ;  w)
\]%%end{equation} 
is an analytic function of $\lambda$ 
throughout the domain  $\mathrm{Int}(\Gamma_{q_0})$.
Thus 
using  \eqref{eq:K-lambda}, 
the Fubini theorem, and the Morera theorem,   it follows that 
for every rectifiable simple closed curve $\Delta$ in $\mathrm{Int}(\Gamma_{q_0})$,
\[%%begin{equation} 
\begin{aligned}
&\int_{\Delta} K_{\lambda} (F;h)\psi)(\xi) d\lambda \\
&= M(\lambda;h)\int_{C_{a,b}'[0,T]}\int_{\mathbb R}\psi(v) \\
& \quad  \times
\bigg(\int_{\Delta}V(\lambda; \xi,v\,; h,w) 
L(\lambda; \xi,v\,;h)
H(\lambda; \xi,v\,;h)
A(\lambda; w)d\lambda\bigg)  dv df(w)\\
&=0.
\end{aligned}
\]%%end{equation}
Therefore
for all $(\xi,  h,  \psi) \in \mathbb R\times (C_{a,b}'[0,T]\backslash\{0\})\times L^1(\mathbb R)$,
$(K_{\lambda} (F;h)\psi)(\xi)$ 
is an analytic function of $\lambda$
throughout the domain $\mathrm{Int}(\Gamma_{q_0})$.
  \end{proof}

%%%%%%%%%%-----------------------------------------------------------%
%%%%\renewcommand{\thetheorem}{\thesection.6}
%%%%%%%%%%-----------------------------------------------------------%  
%%%%%%%%%%-----------------------------------------------------------%
\begin{theorem}\label{thm:main1}
Let $q_0$, $F$, $h$ and $\Gamma_{q_0}$
be as in Lemma \ref{lemma;1-infty}.
Then for each $\lambda\in \mathrm{Int}(\Gamma_{q_0})$, 
the   AOVGFSI $I_{\lambda}^{\mathrm{an}}(F;h)$ exists
and is given by the right-hand side of equation
\eqref{eq:K-lambda}.
Thus,  $K_{\lambda} (F;h)$ 
is an element of $\mathcal{L}(L^1(\mathbb R),L^{\infty}(\mathbb R))$ for each 
$\lambda \in \mathrm{Int}(\Gamma_{q_0})$. 
\end{theorem}
\begin{proof}
Let  $(\lambda,\xi,\psi)\in (0,+\infty)\times \mathbb R \times L^{1}(\mathbb R)$.
We  begin by evaluating the function space integral
\begin{equation}\label{eq:I-lambda-evaluation01}
\begin{aligned}
&(I_{\lambda}(F;h)\psi)(\xi)\\
&=\int_{C_{a,b}[0,T]}F(\lambda^{-1/2}x+\xi)\psi(\lambda^{-1/2}(h,x)^{\sim}+\xi) d \mu (x)\\
&=\int_{C_{a,b}[0,T]}\int_{C_{a,b}'[0,T]}
\exp\{ i\lambda^{-1/2}(w,x)^{\sim} \}\psi(\lambda^{-1/2}(h,x)^{\sim}+\xi)df(w) d \mu (x).
\end{aligned}
\end{equation}
Using the Fubini theorem, we  can change the order of 
integration in \eqref{eq:I-lambda-evaluation01}.
Since $\psi \in L^1(\mathbb R)$,  $f\in \mathcal{M}(C_{a,b}'[0,T])$,
and $(h,x)^{\sim}$ is a Gaussian random variable  
with mean $(h,a)_{C_{a,b}'}$ and variance $\|h\|_{C_{a,b}'}^2$,
it follows  that for $\lambda>0$,
\[%%begin{equation}
\begin{aligned}
 \big|(I_{\lambda}(F;h)\psi)(\xi)\big| 
&\le   \int_{C_{a,b}'[0,T]}\int_{C_{a,b}[0,T]}
  \big|\psi(\lambda^{-1/2}(h,x)^{\sim}+\xi)\big|d\mu(x)d|f|(w)\\
&\le M(|\lambda|;h)
\int_{C_{a,b}'[0,T]}\int_{\mathbb R}|\psi(v)|H(\lambda; \xi,v\,;h)dvd|f|(w)\\
&\le M(|\lambda|;h)
\int_{C_{a,b}'[0,T]}\int_{\mathbb R}|\psi(v)|dud|f|(w)\\
&=M(|\lambda|;h)\|\psi\|_1\|f\|\\
&< +\infty.
\end{aligned}
\]%%end{equation} 
Next, using   \eqref{eq:I-lambda-evaluation01}, the Fubini theorem,   \eqref{eq:GS-expression}, 
\eqref{eq:c-formula}, \eqref{eq:int-formula}, \eqref{eq:simple-K}, \eqref{eq:simple-V}, 
\eqref{eq:simple-L}, \eqref{eq:simple-H}, and \eqref{eq:simple-A}, it follows that
\[%%begin{equation}---------\label{eq:analytic-OP-evaluation}
\begin{aligned}
&(I_{\lambda}(F;h)\psi)(\xi)\\
&=\int_{C_{a,b}'[0,T]}\int_{C_{a,b} [0,T]}
\psi(\lambda^{-1/2}\|h\|_{C_{a,b}'}(e_1,x)^{\sim}+\xi)\\
&\quad \times
\exp\bigg\{ i\lambda^{-1/2} (w,e_1)_{C_{a,b}'}(e_1,x)^{\sim}
 +  i\lambda^{-1/2}\beta_w(e_2(w) ,x)^{\sim} \bigg\}
d\mu(x)df(w)\\
&=\bigg(\frac{\lambda}{2\pi}\bigg) \int_{C_{a,b}'[0,T]}\int_{\mathbb R^2}
\psi( \|h\|_{C_{a,b}'}u_1+\xi)
\exp\bigg\{ i(w,,e_1)_{C_{a,b}'}u_1 
 +  i \beta_wu_2 \\
&\qquad \,\,\,
-\frac{(\sqrt{\lambda}u_1 -(e_1,a)_{C_{a,b}'})^2}{2}
-\frac{(\sqrt{\lambda}u_2 -(e_2(w),a)_{C_{a,b}'})^2}{2}
  \bigg\}du_1du_2df(w)\\
&=\bigg(\frac{\lambda}{2\pi}\bigg)^{1/2} \int_{C_{a,b}'[0,T]}\int_{\mathbb R}
\psi( \|h\|_{C_{a,b}'}u_1+\xi)\\
&\qquad\qquad\qquad \times
\exp\bigg\{ i(w,e_1)_{C_{a,b}'}u_1 
 -\frac{(\sqrt{\lambda}u_1 -(e_1,a)_{C_{a,b}'})^2}{2}\bigg\}du_1\\
&\qquad\qquad\qquad \times
\exp\bigg\{
  -\frac{1}{2\lambda}\beta_w^2 + \frac{i}{\sqrt{\lambda}}\beta_w (e_2(w),a)_{C_{a,b}'}
 \bigg\} df(w)\\
%%%%%%( ***************)
&=M(\lambda;h) \int_{C_{a,b}'[0,T]}\int_{\mathbb R}
\psi(v)
\exp\bigg\{ i\frac{(w,e_1)_{C_{a,b}'}}{\|h\|_{C_{a,b}'}}(v-\xi)\\
&\qquad\qquad\qquad\qquad\qquad
 -\frac{ (\sqrt{\lambda}(v-\xi) -\|h\|_{C_{a,b}'}(e_1,a)_{C_{a,b}'} )^2}{2\|h\|_{C_{a,b}'}^2}\bigg\}dv\\
&\qquad\qquad\qquad\times
\exp\bigg\{
  -\frac{1}{2\lambda}\beta_w^2 + \frac{i}{\sqrt{\lambda}}\beta_w (e_2(w),a)_{C_{a,b}'}
 \bigg\} df(w)\\
%%%%%%( ***************)
&=M(\lambda;h)\int_{C_{a,b}'[0,T]}\int_{\mathbb R}
\psi(v)
V(\lambda; \xi,v\,; h,w)
L(\lambda; \xi,v\,;h)\\
&\qquad\qquad\qquad \qquad\qquad\times
H(\lambda; \xi,v\,;h)
A(\lambda;  w) dv df(w)\\
&=(K_{\lambda} (F;h)\psi)(\xi).
\end{aligned}
\]%%end{equation}
Hence we   see that
the   OVGFSI
$I_{\lambda}(F;h)$ exists for all $(\lambda, h) \in (0,+\infty)\times ( C_{a,b}'[0,T]\backslash\{0\})$.

Let $I_{\lambda}^{\mathrm{an}}(F;h)\psi=K_{\lambda} (F;h)\psi$
for all $\lambda \in \mathrm{Int}(\Gamma_{q_0})$.
Then by Lemma \ref{lemma;1-infty} and Lemma \ref{lemma;2-infty},
we obtain the desired result.
  \end{proof}

%%%%%%%%%%%%%%%%%%%%%%%%%%%%%%%%%%%%%%%%%%%%%%%%%%%%%%%%%%%%%%%%%%
%%%%%%%%%%%%%%%%%%%%%%%%%%%%%%%%%%%%%%%%%%%%%%%%%%%%%%%%%%%%%%%%%%
%%%-----------------------[section]----------------------------%%%
%%%%%%%%%%%%%%%%%%%%%%%%%%%%%%%%%%%%%%%%%%%%%%%%%%%%%%%%%%%%%%%%%%
%%%%%%%%%%%%%%%%%%%%%%%%%%%%%%%%%%%%%%%%%%%%%%%%%%%%%%%%%%%%%%%%%%
\setcounter{equation}{0}
\section{The analytic operator-valued generalized Feynman integral}
\label{sec:AOVGFeynmanI}

In this section we study the AOVG`Feynman'I
$J_q^{\mathrm{an}}(F;h)$ for functionals $F$ in $\mathcal{F}(C_{a,b}[0,T])$.
First of all, we note that for any $q\in \mathbb R\setminus\{0\}$ and any $(\xi,v,h,w)
\in\mathbb R^2\times (C_{a,b}'[0,T]\backslash\{0\})\times C_{a,b}'[0,T]$, 
\[%%begin{equation}
\big|V(-iq;\xi,v\,;h,w)L(-iq;\xi,v\,;h)\big|=1.
\]%%end{equation}
Let $\lambda=-iq\in \widetilde{\mathbb{C}}_+ -\mathbb C_+$.
Then 
\[
\sqrt{\lambda}=\sqrt{-iq}=\sqrt{|q|/2}-i \hbox{sign}(q)\sqrt{|q|/2}.
\]
Hence for $\lambda =-iq$ with $q\in \mathbb R\setminus\{0\}$, 
$[\mathrm{Re}(\sqrt{-iq})]^2 - [\mathrm{Im}(\sqrt{-iq})]^2=0$,
and so 
\[%%begin{equation}
\big|  H(-iq; \xi,v\,;h)\big| 
=\exp\bigg\{  \frac{\sqrt{2|q|}(h,a)_{C_{a,b}'}
(v-\xi)-(h,a)_{C_{a,b}'}^2}{2\|h\|_{C_{a,b}'}^2}    \bigg\}
\]%%end{equation}
which is not necessarily in $L^p(\mathbb R)$, as a function of $v$, for any $p\in[1,+\infty]$.
Hence $K_{-iq}(F;h)$ might not exist as an element of 
$\mathcal{L}(L^1(\mathbb R),L^{\infty}(\mathbb R))$.

Let $q=-1$ and let $h$ be an element of $C_{a,b}'[0,T]$ with $\|h\|_{C_{a,b}'}=1$
and  with $(h,a)_{a,b}>0$ (we can choose $h$ to be $a/\|a\|_{C_{a,b}'}$ in $C_{a,b}'[0,T]$).
Let $\psi: \mathbb R\to \mathbb C$ be defined by the formula
\[%%begin{equation}-----\label{eq:psi-example}
\begin{aligned}
\psi(v)
&= v \chi_{[0,+\infty)}(v)\\
&\quad\times
\exp\bigg\{\frac{iv^2}{2}-\frac{i\sqrt2 (h,a)_{C_{a,b}'}v}{2}
+\frac{(h,a)_{C_{a,b}'}^2}{2}-\frac{ \sqrt2 (h,a)_{C_{a,b}'}v}{4}\bigg\}.
\end{aligned}
\]%%end{equation}
We note that
\[%%begin{equation}
|\psi(v)|=v\chi_{[0,+\infty)}(v)
\exp\bigg\{\frac{(h,a)_{C_{a,b}'}^2}{2}-\frac{ \sqrt2 (h,a)_{C_{a,b}'}v}{4}\bigg\},
\]%%end{equation}
and hence $\psi  \in L^p(\mathbb R)$ for all $p\in [1,+\infty]$.
In fact, $\psi$ is also an element of $C_0(\mathbb R)$,
the space of bounded continuous functions on $\mathbb R$
that vanish at infinity.

Let $F(x)\equiv 1$. Then $F$ is an element of $\mathcal{F}^{q_0}$
for all $q_0\in (0,+\infty)$, and 
$(K_{-iq}(F;h)\psi)(\xi)$ with $q=-1$ is given by 
\begin{equation}\label{eq:K-1}
(K_{ i }(1;h)\psi)(\xi)
=\bigg(\frac{i}{2\pi}\bigg)^{1/2}\int_{\mathbb R}\psi(v)H(i;\xi,v\,; h) dv.
\end{equation}

Next, using equation \eqref{eq:anal-H} with $\lambda=i$ and $\sqrt{\lambda}=\sqrt{i}=(1+i)/\sqrt2$,
we observe that
\[%%begin{equation}-----\label{eq:H-example}
\begin{aligned}
&H(i;\xi,v\,;h)\\
&=\exp\bigg\{-i\frac{(v-\xi)^2}{2}+\frac{ (h,a)_{C_{a,b}'}(v-\xi)}{\sqrt2}
+\frac{i (h,a)_{C_{a,b}'}(v-\xi)}{\sqrt2}-\frac{(h,a)_{C_{a,b}'}^2}{2}\bigg\},
\end{aligned}
\]%%end{equation}
and hence, 
\begin{equation}\label{eq:H-example}
\begin{aligned}
&\psi(v)H(i;\xi,v\,;h)\\
&=v\chi_{[0,+\infty)}(v)\exp\bigg\{\frac{ \sqrt2 (h,a)_{C_{a,b}'}v}{4}
  +i\xi v -\frac{i\xi^2}{2}-\bigg(\frac{1+i}{\sqrt2}\bigg)(h,a)_{C_{a,b}'}\xi\bigg\}
\end{aligned}
\end{equation}
which is not an element of $L^p(\mathbb R)$,
as a function of $v$, for any $p\in[1,+\infty]$.

Then, using equations \eqref{eq:K-1}   and  \eqref{eq:H-example},
we see that
\[%%begin{equation}
\begin{aligned}
 (K_{i}(1;h)\psi)(\xi) 
&=\bigg(\frac{i}{2\pi}\bigg)^{1/2}
\exp\bigg\{-\frac{i\xi^2}{2}-\bigg(\frac{1+i}{\sqrt2}\bigg)(h,a)_{C_{a,b}'}\xi\bigg\}\\
&\qquad \times
\int_{\mathbb R}
v\chi_{[0,+\infty)}(v)\exp\bigg\{ 
\frac{ \sqrt2 (h,a)_{C_{a,b}'}v}{4}
    +i\xi v  \bigg\}dv.
\end{aligned}
\]%%end{equation}
Hence, choosing $\xi=0$, and using the fact that $(h,a)_{C_{a,b}'}$ is positive,
we see that 
\[%%begin{equation}
\big|(K_{ i }(1;h)\psi)(0)\big|
=(2\pi )^{-1/2}
\int_{0}^{+\infty}v\exp\bigg\{ \frac{ \sqrt2 (h,a)_{C_{a,b}'}v}{4}\bigg\}dv
=+\infty.
\]%%end{equation}
In fact, for each fixed $\xi\in \mathbb R$, we observe that
\[%%begin{equation}
\begin{aligned}
&\big|(K_{i}(1;h)\psi)(\xi)\big|\\
&=(2\pi)^{-1/2}
\exp\bigg\{ - \frac{1 }{\sqrt2} (h,a)_{C_{a,b}'}\xi\bigg\}\\
&\quad\times
\bigg|\int_{\mathbb R}
v\chi_{[0,+\infty)}(v)\exp\bigg\{ 
\frac{ \sqrt2 (h,a)_{C_{a,b}'}v}{4}
    +i\xi v  \bigg\}dv\bigg|\\
&=+\infty,
\end{aligned}
\]%%end{equation}
and so $(K_{i}(1;h)\psi)$ is not an element of $L^{\infty}(\mathbb R)$
even though $\psi$ was an element of $L^1(\mathbb R)$.
Hence $K_{-iq}(F;h)\psi \equiv K_{i}(1;h)\psi$ is not in $\mathcal{L}(L^1(\mathbb R),L^{\infty}(\mathbb R))$.

In this section,   we thus clearly  need to impose additional restrictions on $\psi$
for the existence of our AOVG`Feynman'I.

For any  positive real number $\delta$,
let $\nu_{\delta,a}$ be a   measure on $\mathbb R$ with 
\[%%begin{equation}
d \nu_{\delta,a}=\exp\{\delta\mathrm{Var}(a)u^2\}du
\]%%end{equation}
where $\mathrm{Var}(a)=|a|(T)$ denotes the total variation of $a$,
the mean function of  the GBMP, on $[0,T]$ 
and let $L^1(\mathbb R,\nu_{\delta,a})$ be the space of $\mathbb C$-valued Lebesgue 
measurable functions $\psi$
on $\mathbb R$ such that $\psi$ is integrable with respect to the  measure $\nu_{\delta,a}$ on $\mathbb R$.
Let $\|\cdot\|_{1,\delta}$  denote the 
  $L^{1}(\mathbb R,\nu_{\delta,a})$-norm.
Then for all $\delta>0$, we have the following  inclusion 
\begin{equation}\label{eq:triple-1}
L^{1}(\mathbb R,\nu_{\delta,a}) \subsetneq L^{1}(\mathbb R)
\end{equation}
as sets, because $\|\psi\|_{1}\le \|\psi\|_{1,\delta}$ for all $\psi\in L^{1}(\mathbb R)$.

Let $\mathcal{L}(L^1(\mathbb R,\nu_{\delta,a}),L^{\infty}(\mathbb R))$ 
be the space of continuous linear operators form 
$L^1(\mathbb R,\nu_{\delta,a})$ to $L^{\infty}(\mathbb R)$.
In Theorem \ref{thm:main1}, we proved that 
for all $\psi\in L^1(\mathbb R)$,
$I_{\lambda}^{\mathrm{an}}(F;h)\psi$ is in $L^{\infty}(\mathbb R)$.
From the inclusion \eqref{eq:triple-1}, we see that 
for all   $\psi\in L^1(\mathbb R,\nu_{\delta,a})$,
$I_{\lambda}^{\mathrm{an}}(F;h)\psi$ is in $L^{\infty}(\mathbb R)$.
Furthermore, for all $\delta>0,$
\begin{equation}\label{space-inclusion}
\mathcal{L}(L^1(\mathbb R), L^{\infty}(\mathbb R))
\subset
\mathcal{L}(L^1(\mathbb R,\nu_{\delta,a}), L^{\infty}(\mathbb R)),
\end{equation}
as sets. 

Now, the notation $\|\cdot\|_{{\rm o},\delta} $
will be used for the norm  on  $\mathcal{L}(L^1(\mathbb R,\nu_{\delta,a}),L^{\infty}(\mathbb R))$.
We already shown  in \eqref{eq:add-number001} that for
all $(\lambda,\xi,\psi)\in \mathrm{Int}(\Gamma_{q_0})\times\mathbb R\times L^1(\mathbb R)$,
\[%%begin{equation}
\begin{aligned}
&\big|(K_{\lambda}(F;h)\psi)(\xi)\big|\\
&\le M(|\lambda|;h)\int_{\mathbb R}\big|\psi(v)\big|\big|H(\lambda;\xi,v\,;h)\big| dv 
\int_{C_{a,b}'[0,T]}\big|A(\lambda;w)\big|d |f|(w).
\end{aligned}
\]%%end{equation}
But, by the same method, \eqref{eq:add-500}, and \eqref{2020-a2}, it also follows  that 
for any $\delta>0$ and all $(\lambda,\xi,\psi)\in \mathrm{Int}(\Gamma_{q_0})\times\mathbb R\times L^1(\mathbb R,\nu_{\delta,a})$,
\begin{equation}\label{eq;observation01}
\begin{aligned}
&\big|(K_{\lambda}(F;h)\psi)(\xi)\big|\\
&\le 
M(|\lambda|;h)\int_{\mathbb R}\big|\psi(v)\big|\big|H(\lambda;\xi,v\,;h)\big| dv 
\int_{C_{a,b}'[0,T]}\big|A(\lambda;w)\big|d |f|(w)\\
&\le 
M(|\lambda|;h)\int_{\mathbb R}\big|\psi(v)
\exp\{\delta \mathrm{Var}(a) v^2\}\big|\big|H(\lambda;\xi,v\,;h) \big| dv \\
&\qquad\qquad\,\,\,\times
\int_{C_{a,b}'[0,T]}\big|A(\lambda;w)\big|d |f|(w)\\
&\le
 M(|\lambda|;h)S(\lambda;h)\int_{\mathbb R}\big|\psi(v)\exp\{\delta \mathrm{Var}(a) v^2\}\big|dv \\
&\qquad\qquad\,\,\,\times
\int_{C_{a,b}'[0,T]}k(q_0;w)d |f|(w)\\
&\le
\|\psi\|_{1,\delta}\bigg(S(\lambda;h)M(|\lambda|;h)
\int_{C_{a,b}'[0,T]}k(q_0;w)d |f|(w)\bigg)\\
\end{aligned}
\end{equation}
and so 
\[%%begin{equation}
\|K_{\lambda}(F;h)\|_{{\rm o},\delta} 
\le
S(\lambda;h)M(|\lambda|;h)
\int_{C_{a,b}'[0,T]}k(q_0;w)d |f|(w).
\]%%end{equation}
Thus we have the following definition.

%%%%%%%%%%-----------------------------------------------------------%
%%%%\renewcommand{\thedefinition}{\thesection.1}
%%%%%%%%%%-----------------------------------------------------------%  
%%%%%%%%%%-----------------------------------------------------------%
\begin{definition}\label{def:op-basic3}
Let $\Gamma$ be as in Definition \ref{def:op-basic2} and 
let $q$ be a nonzero real number  with $-iq \in \Gamma$.
Suppose that there exists an operator 
$J_{q}^{\mathrm{an}} (F;h)$ in 
$\mathcal{L}(L^1(\mathbb R,\nu_{\delta,a}),L^{\infty}(\mathbb R))$
for some $\delta>0$ such that 
for every $\psi$ in $L^1(\mathbb R,\nu_{\delta,a})$,
\[%%begin{equation}
\big\|J_q^{\mathrm{an}}(F;h)\psi - I_{\lambda}^{\mathrm{an}}(F;h)\psi\big\|_{\infty} \to 0
\]%%end{equation}
as $\lambda \to -iq$ through $\hbox{\rm Int}(\Gamma) $, then $J_{\lambda}^{\mathrm{an}}(F;h)$ is called the 
AOVG`Feynman'I of $F$ with parameter $q$.
\end{definition}

%%%%%%%%%%-----------------------------------------------------------%
%%%%\renewcommand{\thetheorem}{\thesection.2}
%%%%%%%%%%-----------------------------------------------------------%  
%%%%%%%%%%-----------------------------------------------------------%
\begin{theorem}\label{thm:main2}
Let $q_0$, $F$,  $h$ and  $\Gamma_{q_0}$ be
as in Lemma \ref{lemma;1-infty}.
Then for all real $q$ with $|q|>q_0$,
the   AOVG`Feynman'I of $F$, $J_{q}^{\mathrm{an}}(F;h)$,
 exists as an element of
$\mathcal{L}(L^1(\mathbb R,\nu_{\delta,a}),L^{\infty}(\mathbb R ))$
for any $\delta>0$,
and is given by the right-hand side of equation 
\eqref{eq:K-lambda}  with $\lambda=-iq$.
\end{theorem}
\begin{proof}
First, we will show that 
$K_{-iq}(F;h)$ is an element of 
$\mathcal{L}(L^1(\mathbb R,\nu_{\delta,a}), L^{\infty}(\mathbb R ))$.
Note that for every $\delta>0$,
$|  H(-iq; \xi,v\,;h)|\exp\{-\delta\mathrm{Var}(a) u^2\}$ is  bounded by $1$.
Hence  
for any $\delta\in (0,+\infty)$ and every $\psi \in  L^1(\mathbb R,\nu_\delta)$,
\[%%begin{equation}
\begin{aligned}
&\int_{\mathbb R}\big|\psi(v)\big|\big|  H(-iq; \xi,v\,;h)\big|  d v\\
&=\int_{\mathbb R} \big |\psi(v)\big|\exp\big\{\delta\mathrm{Var}(a) u^2\big\}
   \big|  H(-iq; \xi,v\,;h)\big|  \exp\big\{-\delta\mathrm{Var}(a) u^2\big\}d v\\
&\le \|\psi\|_{1, {\delta}}.
\end{aligned}
\]%%end{equation}
Also, by a simple calculation, it follows that
\[%%begin{equation}
|V(-iq;\xi,v\,;h,w)\big|\big|L(-iq;\xi,v\,;h)\big|=1.
\]%%end{equation}
Thus, using these and \eqref{2020-a2}, it also follows that
for all real $q$ with $|q|>q_0$,
\begin{equation}\label{eq;observation02}
\begin{aligned}
&\big|(K_{-iq}(F;h)\psi)(\xi)\big|\\
&\le M(|q|;h)\int_{C_{a,b}'[0,T]}\int_{\mathbb R}
 \big|\psi(v)\big|\big|V(-iq;\xi,v\,;h,w)\big|\big|L(-iq;\xi,v\,;h)\big|\\
&\qquad \qquad \qquad \qquad \qquad \qquad \times 
  \big|H(-iq;\xi,v\,;h)\big|  \big|A(-iq;w)\big| d v d|f|(w)\\
&= M(|q|;h) \int_{\mathbb R}  \big|\psi(v)\big|
  \big|H(-iq;\xi,v\,;h)\big| d v
\int_{C_{a,b}'[0,T]} \big|A(-iq;w)\big|   d|f|(w)\\
&\le  \|\psi\|_{1, {\delta}}\bigg( M(|q|;h) 
\int_{C_{a,b}'[0,T]}k(q_0;w) d|f|(w)\bigg).\\
\end{aligned}
\end{equation}
Therefore we have that
\[%%begin{equation}
\big\| K_{-iq}(F;h)\psi \big\|_{\infty }\le \|\psi\|_{1, {\delta}}
\bigg(M(|q|;h)\int_{C_{a,b}'[0,T]}k(q_0;w) d|f|(w)\bigg)
\]%%end{equation}
and 
\[%%begin{equation}
\big\| K_{-iq}(F;h)  \big\|_{{\rm o},\delta}
\le  
M(|q|;h)\int_{C_{a,b}'[0,T]}k(q_0;w) d|f|(w),
\]%%end{equation}
and implies that $K_{-iq}(F;h)\in \mathcal{L}(L^1(\mathbb R,\nu_{\delta,a}), L^{\infty}(\mathbb R ))$.

We  now want to show that the   AOVG`Feynman'I $J_{q}^{\mathrm{an}}(F;h)$  of $F$
exists and is given by the right-hand side of  \eqref{eq:K-lambda} with $\lambda=-iq$.
To do this, it suffices to show that 
for every $\psi$ in $L^1(\mathbb R,\nu_{\delta,a})$
\[%%begin{equation}
\big\| 
    K_{-iq}(F;h)\psi -  I_{\lambda}^{\mathrm{an}}(F;h)\psi\big\|_{{\infty} }\to 0
\]%%end{equation}
as $\lambda\to-iq$ through $\mathrm{Int}(\Gamma_{q_0})$, where $\Gamma_{q_0}$ is
given by equation \eqref{eq:domain}.
But, in view of Lemmas \ref{lemma;1-infty}, \ref{lemma;2-infty},
Theorem \ref{thm:main1}, and equation  \eqref{space-inclusion},
we already proved that $I_{\lambda}^{\mathrm{an}}(F;h)= K_{\lambda} (F;h)$
for all $\lambda\in \mathrm{Int}(\Gamma_{q_0})$
and that $K_{\lambda} (F;h)$ is an element of 
$\mathcal{L}(L^1(\mathbb R,\nu_{\delta,a}),L^{\infty}(\mathbb R ))$.
Next, by \eqref{eq;observation01} and
\eqref{eq;observation02}, we obtain that
for all $(\lambda,\xi, \psi)\in \Gamma_{q_0}\times \mathbb R  \times L^1(\mathbb R,\nu_\delta)$,
\[%%begin{equation}
\begin{aligned}
& \big|(K_{\lambda}(F;h)\psi) (\xi)\big|\\
&\le
\begin{cases}
\|\psi\|_{1,\delta}\big\{S(\lambda;h)M(|\lambda|;h)
\int_{C_{a,b}'[0,T]}k(q_0;w)d |f|(w)\big\},  \!\!\!\!&\lambda \in \hbox{\rm Int}(\Gamma_{q_0})\\
\|\psi\|_{1,\delta}\big\{ M(|q|;h)\int_{C_{a,b}'[0,T]}k(q_0;w) d|f|(w)\big\}, 
   &\lambda=-iq,\, q\in \mathbb R\backslash\{0\}
\end{cases}\\
&<+\infty.
\end{aligned}
\]%%end{equation}
Moreover, using  the techniques similar to those used in
the proof of Lemma \ref{lemma;2-infty}, one 
can easily verify  that  
there exists a sufficiently small $\varepsilon_0>0$
satisfying the inequality: 
\[%%begin{equation}
\begin{aligned}
& \big|(K_{\lambda}(F;h)\psi) (\xi)\big|\\
&\le
\|\psi\|_{1,\delta}
\Bigg(
\exp\Bigg\{\frac{(h,a)_{C_{a,b}'}^2}{4\|h\|_{C_{a,b}'}^2 } \bigg(\frac{q_0}{\varepsilon_0}  +1\bigg)\Bigg\}
%%%\\&\qquad\qquad  \times
M(1+|q|;h)\int_{C_{a,b}'[0,T]}k(q_0;w) d|f|(w)\Bigg)\\ 
  &<+\infty 
\end{aligned}
\]%%end{equation}
for all $\lambda \in
\Gamma_{q_0} \cap \{\lambda\in \widetilde{\mathbb C}: |\lambda-(-iq )|<\varepsilon_0\}$. 
Hence by the dominated convergence theorem, we have
\[%%begin{equation}
\lim_{\lambda\to -iq}(I_{\lambda}^{\mathrm{an}}(F;h)\psi)(\xi)
=\lim_{\lambda\to -iq}(K_{\lambda} (F;h)\psi)(\xi)
=(K_{-iq} (F;h)\psi)(\xi)\\
\]%%end{equation}
for each $\xi \in \mathbb R$.
Thus $J_{q}^{\mathrm{an}}(F;h)$ exists as  an element of 
$\mathcal{L}(L^1(\mathbb R,\nu_{\delta,a}), L^{\infty}(\mathbb R ))$
and 
is given by the right-hand side of equation 
\eqref{eq:K-lambda}  with $\lambda=-iq$.
  \end{proof}

It is clear that given two positive real number $\delta_1$ and $\delta_2$ with $\delta_1<\delta_2$,
\[%%begin{equation}
L^1(\mathbb R, \nu_{\delta_2,a})  \subsetneq L^1(\mathbb R, \nu_{\delta_1,a})  \subsetneq L^1(\mathbb R).
\]%%end{equation}
Thus it follows that
\[%%begin{equation}
\mathcal L(L^1(\mathbb R),L^{\infty}(\mathbb R)) 
\subsetneq  \mathcal L( L^1(\mathbb R, \nu_{\delta_1,a}),L^{\infty}(\mathbb R))
\subsetneq  \mathcal L( L^1(\mathbb R, \nu_{\delta_2,a}),L^{\infty}(\mathbb R)).
\]%%end{equation}
Let
\[%%begin{equation}
L^{1,a}(\mathbb R)=\bigcup_{\delta>0} L^1(\mathbb R,\nu_{\delta,a})
\]%%end{equation}
and let
\[%%begin{equation}
\mathfrak B( L^{1,a}(\mathbb R),L^{\infty}(\mathbb R)) =\bigcap_{\delta>0} \mathcal L( L^1(\mathbb R,\nu_{\delta,a}),L^{\infty}(\mathbb R)).
\]%%end{equation}
We note that $L^{1,a}(\mathbb R)$ and $\mathfrak B( L^{1,a}(\mathbb R),L^{\infty}(\mathbb R))$
are not normed spaces.
However we can suggest set theoretic structures between themselves as follows:
since  $L^1(\mathbb R,\nu_{\delta,a}) \subset L^{1,a}(\mathbb R)\subset L^1(\mathbb R)$
for any $\delta>0$,  it follows that
\[%%begin{equation}
\mathcal L( L^1(\mathbb R),L^{\infty}(\mathbb R)) 
\subset \mathfrak B ( L^{1,a}(\mathbb R),L^{\infty}(\mathbb R)) 
\subset  \mathcal L( L^1(\mathbb R,\nu_{\delta,a}),L^{\infty}(\mathbb R)).
\]%%end{equation}

From  this observation and Theorem \ref{thm:main2},
we can obtain the following assertion.

%%%%%%%%%%-----------------------------------------------------------%
%%%%\renewcommand{\thetheorem}{\thesection.3}
%%%%%%%%%%-----------------------------------------------------------%  
%%%%%%%%%%-----------------------------------------------------------%
\begin{theorem}\label{thm:main2020}
Let $q_0$, $F$,  $h$ and  $\Gamma_{q_0}$ be
as in Lemma \ref{lemma;1-infty}.
Then for all real $q$ with $|q|>q_0$,
the   AOVG`Feynman'I $J_{q}^{\mathrm{an}}(F;h)$ exists as an element of
$\mathfrak B(  L^{1,a}(\mathbb R),L^{\infty}(\mathbb R))$.
\end{theorem}

%%%%%%%%%%-----------------------------------------------------------%
%%%%\renewcommand{\theremark}{\thesection.4}
%%%%%%%%%%-----------------------------------------------------------%  
%%%%%%%%%%-----------------------------------------------------------%
\begin{remark}
If $b(t)=t$ and $a(t)\equiv 0$ on $[0,T]$, the function space $C_{a,b}[0,T]$ reduces to the classical Wiener space
$C_0[0,T]$. In this case, the three linear spaces  $L^1(\mathbb R)$, $L^1(\mathbb R, \nu_{\delta,0})$ and
$L^{1,0}(\mathbb R)$ coincide each other.
Furthermore, the three classes
$\mathcal L( L^1(\mathbb R),L^{\infty}(\mathbb R))$, 
$\mathfrak B ( L^{1,0}(\mathbb R),L^{\infty}(\mathbb R))$, and 
$\mathcal L( L^1(\mathbb R,\nu_{\delta,0}),L^{\infty}(\mathbb R))$
also coincide.
\end{remark}

%%%%%%%%%%%%%%%%%%%%%%%%%%%%%%%%%%%%%%%%%%%%%%%%%%%%%%%%%%%%%%%%%%
%%%%%%%%%%%%%%%%%%%%%%%%%%%%%%%%%%%%%%%%%%%%%%%%%%%%%%%%%%%%%%%%%%
%%%-----------------------[section]----------------------------%%%
%%%%%%%%%%%%%%%%%%%%%%%%%%%%%%%%%%%%%%%%%%%%%%%%%%%%%%%%%%%%%%%%%%
%%%%%%%%%%%%%%%%%%%%%%%%%%%%%%%%%%%%%%%%%%%%%%%%%%%%%%%%%%%%%%%%%%
\setcounter{equation}{0}
\section{Examples}

In this section, we present interesting examples to which our results  
in previous sections can be applied.

\par
Let $\mathcal{M}(\mathbb R)$ be the class of   complex-valued, 
countably additive measures on $\mathcal{B}(\mathbb R)$.
For $\eta\in \mathcal{M}(\mathbb R)$, the Fourier  transform 
$\widehat{\eta}$ of $\eta$
is a $\mathbb C$-valued function defined on $\mathbb R$, given by the formula
\[%%begin{equation}
\widehat{\eta}(u)=\int_{\mathbb R}\exp\{i uv\} d \eta(v).
\]%%end{equation}

\par 
(1)   Let $w_0\in C_{a,b}'[0,T]$ and let $\eta \in \mathcal{M}(\mathbb R)$.
Define $F_1:C_{a,b}[0,T]\to \mathbb C$ by
\[%%begin{equation}\label{eq:ex-F1}
F_1(x)=\widehat{\eta}((w_0,x)^{\sim}).
\]%%end{equation}
Define a function  $\phi:\mathbb R\to C_{a,b}'[0,T]$ 
by $\phi(v)=vw_0$.
Let $f=\eta\circ \phi^{-1}$.
It is quite clear  that $f$ is in $\mathcal{M}(C_{a,b}'[0,T])$
and is supported by $[w_0]$, the subspace of $C_{a,b}'[0,T]$ 
spanned by $\{w_0\}$.
Now  for s-a.e. $x\in C_{a,b}[0,T]$,
\[%%begin{equation}
\begin{aligned}
\int_{C_{a,b}'[0,T]}\exp\{i(w,x)^{\sim}\}d f(w) 
&=\int_{C_{a,b}'[0,T]}\exp\{i(w,x)^{\sim}\}d (\eta\circ\phi^{-1})(w)\\
&=\int_{\mathbb R}\exp\{i(\phi(v),x)^{\sim}\}d \eta (v)\\
&=\int_{\mathbb R}\exp\{i (w_0,x)^{\sim}v\}d \eta (v)\\
&=\widehat{\eta}((w_0,x)^{\sim}).
\end{aligned}
\]%%end{equation}
Thus $F_1$ is an element of    $\mathcal{F}(C_{a,b}[0,T])$.

\par
Suppose that for a  fixed positive real number $q_0>0$,
\begin{equation}\label{eq:condi-ex}
\int_{\mathbb R}
\exp\big\{ (2q_0)^{-1/2}\|w_0\|_{C_{a,b}'}\|a\|_{C_{a,b}'}|v|\big\}
d |\eta|(v)<+\infty.
\end{equation}
It is easy to show that condition \eqref{eq:condi-ex} 
is equivalent to   condition \eqref{eq:condition-finite} with $f=\eta\circ \phi^{-1}$.
Thus, under condition \eqref{eq:condi-ex},
 $F_1$ is an element of $\mathcal{F}^{q_0}$ and so,
by Theorem \ref{thm:main2}, $J_{q}^{\mathrm{an}}(F_1;h)$
exists as an element of $\mathcal{L}(L^1(\mathbb R,\nu_{\delta,a}),L^{\infty}(\mathbb R))$
for all real $q$ with $|q|>q_0$, all $h\in C_{a,b}'[0,T]\backslash\{0\}$,
 and any $\delta>0$. Moreover  $J_{q}^{\mathrm{an}}(F_1;h)$ is 
an element of the space 
$\mathfrak{B}(  L^{1,a}(\mathbb R),L^{\infty}(\mathbb R))$ 
by Theorem \ref{thm:main2020}.

 \par
Next, we present more explicit examples of functionals in $\mathcal{F}(C_{a,b}[0,T])$
whose associated measures satisfy   condition \eqref{eq:condi-ex}.

(2)   Let $S:C_{a,b}'[0,T]\to C_{a,b}'[0,T]$ 
be the linear operator defined by  $Sw(t)=\int_0^t w(s)db(s)$.
Then the adjoint operator $S^*$ of $S$ is given by
\[%%begin{equation}
S^*w(t)=\int_0^t \big(w(T)-w(s)\big) d b(s)
\]%%end{equation}
and for $x\in C_{a,b}[0,T]$, $(S^*b,x )^{\sim} =\int_0^T x(t) d b(t)$
by an integration by parts formula.

Given $\mathrm{m}$ and $\sigma^2$  in $\mathbb R$
with $\sigma^2>0$,
let $\eta_{\mathrm{m},\sigma^2}$ be the Gaussian measure
given by
\begin{equation}\label{eq:measure-01}
\eta_{\mathrm{m},\sigma^2}(B)
=( 2\pi \sigma^2)^{-1/2}
\int_{B}\exp\bigg\{- \frac{(v -\mathrm{m} )^2}{2\sigma^2}\bigg\}d v,
\quad B \in \mathcal{B}(\mathbb R).
\end{equation}
Then $\eta_{\mathrm{m},\sigma^2 }\in \mathcal{M}(\mathbb R )$
and
\[%%begin{equation}
\widehat {\eta_{\mathrm{m}, \sigma^2} }( u)
=\int_{\mathbb R}\exp\{ i uv\}d\eta_{\mathrm{m},\sigma^2}(v)
=\exp\bigg\{-\frac12 \sigma^2u^2+i\mathrm{m} u \bigg\}.
\]%%end{equation}
The complex measure $\eta_{m,\sigma^2}$ given by equation \eqref{eq:measure-01}
satisfies condition  \eqref{eq:condi-ex}  for all $q_0>0$.
Thus we can apply the results in argument (1) 
to the functional $F_2:C_{a,b}[0,T]\to\mathbb C$ given by
\begin{equation}\label{eq:example-special}
\begin{aligned}
F_2(x)
&=\widehat {\eta_{\mathrm{m}, \sigma^2} }( (w_0,x)^{\sim})\\
&=\exp\bigg\{-\frac12 \sigma^2[(w_0,x)^{\sim}]^2+i\mathrm{m}(w_0,x)^{\sim}\bigg\}.
\end{aligned}\end{equation} 
For example, if we choose  $w_0=S^*b$, $\mathrm{m}=0$ and $\sigma^2=2$
in \eqref{eq:example-special},   
we have 
\[%%begin{equation}\label{eq:F2+1}
F_3(x)
=\exp\big\{- [(S^*b,x)^{\sim}]^2 \big\}
=\exp\bigg\{-\bigg(\int_0^T x(t)d b(t)\bigg)^2\bigg\}
\]%%end{equation}
for $x\in C_{a,b}[0,T]$.

We note that the functional $F_3$ is in $\cup_{q_0>0}\mathcal F^{q_0}$,
and so that for every nonzero real number $q$, the  
 AOVG`Feynman'I $J_{q}^{\mathrm{an}}(F_3;h)$ exists as an element of
$\mathfrak B(  L^{1,a}(\mathbb R),L^{\infty}(\mathbb R))$.

\par
(3)
Let $F_4:C_{a,b}[0,T]\to\mathbb C$ be given by
\[%%begin{equation}\label{eq:F3+1}
F_4(x)=\exp\bigg\{i\int_0^Tx(t) d b(t)\bigg\}. 
\]%%end{equation}.
Then $F_4$ is a functional in $\mathcal F(C_{a,b}[0,T])$,
because
\[%%begin{equation}
F_4(x)
 =\exp\{i(S^*b,x)^{\sim}\} 
=\int_{C_{a,b}'[0,T]}\exp\{i(w,x)^{\sim} \}d \zeta (w)
\]%%end{equation}
for s-a.e. $x\in C_{a,b}[0,T]$, where $\zeta$ is the Dirac  measure concentrated at $S^*b$ in $C_{a,b}'[0,T]$.
The   Dirac  measure $\zeta$ also satisfies condition 
\eqref{eq:condition-finite} with $f$ replaced with $\zeta$ for all $q_0>0$;
that is, $F_4 \in \cup_{q_0>0}\mathcal F^{q_0}$,
and so that for every nonzero real number $q$, the  
 AOVG`Feynman'I $J_{q}^{\mathrm{an}}(F_4;h)$ exists as an element of
$\mathfrak B(  L^{1,a}(\mathbb R),L^{\infty}(\mathbb R))$.

\section*{Statements and Declarations}

\subsection*{Funding}
This work was supported by the National Research Foundation of Korea(NRF) 
grant funded by the Korea government(MSIT) (No. 2021R1F1A1062770).

\subsection*{Conflicts of Interest} The author declares no conflict of interest.

\subsection*{Data Availability Statement}
No supplementary material is available.

\subsection*{Compliance with ethical standards}
 Not applicable.

%%%%%%%%%%%%%%%%%%%%%%%%%%%%%%%%%%%%%%%%%%%%%
%%%%%%%%%%%%%%%%%%%%%%%%%%%%%%%%%%%%%%%%%%%%%
%%%%%%%%%%%%%%%%%%%%%%%%%%%%%%%%%%%%%%%%%%%%%
%%%%%%%%%%%%%%%%%%%%%%%%%%%%%%%%%%%%%%%%%%%%%
%%%%%%%%%%%%%%%%%%%%%%%%%%%%%%%%%%%%%%%%%%%%%
%%%%%%%%%%%                 %%%%%%%%%%%%%%%%%
%%%%%%%%%%%   references    %%%%%%%%%%%%%%%%%
%%%%%%%%%%%                 %%%%%%%%%%%%%%%%%
%%%%%%%%%%%%%%%%%%%%%%%%%%%%%%%%%%%%%%%%%%%%%
%%%%%%%%%%%%%%%%%%%%%%%%%%%%%%%%%%%%%%%%%%%%%
%%%%%%%%%%%%%%%%%%%%%%%%%%%%%%%%%%%%%%%%%%%%%
%%%%%%%%%%%%%%%%%%%%%%%%%%%%%%%%%%%%%%%%%%%%%
%%%%%%%%%%%%%%%%%%%%%%%%%%%%%%%%%%%%%%%%%%%%%
%%%%%\bibliographystyle{amsplain}

\end{document}